\documentclass[reqno]{amsart}
\usepackage{amsmath,amsfonts,amssymb,color}
\newtheorem{teor}{Theorem}[section]
\newtheorem{lema}[teor]{Lemma}
\newtheorem{prop}[teor]{Proposition}

\newtheorem{coro}[teor]{Corollary}
\theoremstyle{definition}
\newtheorem{defi}[teor]{Definition}
\newtheorem{eje}[teor]{Example}

\newtheorem{nota}[teor]{Remark}

\numberwithin{equation}{section}

\newcommand{\R}{\mathbb{R}}

\newcommand{\A}{\mathbb{A}}
\newcommand{\B}{\mathbb{B}}

\newcommand{\T}{\mathbb{T}}

\newcommand{\dd}{\textsf{d}}

\newcommand{\wit}{\widetilde}

\newcommand{\lsm}{\left[\!\begin{smallmatrix}}
\newcommand{\rsm}{\end{smallmatrix}\!\right]}

\newcommand{\des}{\displaystyle}

\DeclareMathOperator{\cls}{cls} 
\DeclareMathOperator{\Int}{Int}

\begin{document}
\title[Forwards attraction properties in linear-dissipative parabolic PDEs]
{Forwards attraction properties in scalar non-autonomous  linear-dissipative parabolic PDEs. The case of null upper Lyapunov exponent}
\author[J.A. Langa]{Jos\'{e} A. Langa}
\author[R. Obaya]{Rafael Obaya}
\author[A.M. Sanz]{Ana M. Sanz}
\address[J.A. Langa]{Departamento de Ecuaciones Diferenciales y An\'{a}lisis Num\'{e}rico, Universidad de Sevilla, C/ Tarfia s/n  41012 Sevilla, Spain}  \email{langa@us.es}
\address[R. Obaya]{Departamento de Matem\'{a}tica
Aplicada, E. Ingenier\'{\i}as Industriales, Universidad de Valladolid,
47011 Valladolid, Spain, and member of IMUVA, Instituto de Investigaci\'{o}n en
Matem\'{a}ticas, Universidad de Valladolid, Spain.}
 \email{rafoba@wmatem.eis.uva.es}
\address[A.M. Sanz]{Departamento de Did\'{a}ctica de las Ciencias Experimentales, Sociales y de la Matem\'{a}tica,
Facultad de Educaci\'{o}n, Universidad de Valladolid, 34004 Palencia, Spain,
and member of IMUVA, Instituto de Investigaci\'{o}n en  Mate\-m\'{a}\-ti\-cas, Universidad de
Valladolid.} \email{anasan@wmatem.eis.uva.es}
\thanks{R. Obaya and A.M. Sanz were partly supported by FEDER Ministerio de Econom\'{\i}a y Competitividad grant
MTM2015-66330-P, and the European Commission under project H2020-MSCA-ITN-2014 643073 CRITICS}
\thanks{J.A.  Langa was partially supported by  Junta de Andaluc\'{\i}a under Proyecto de Excelencia FQM-1492 and  FEDER Ministerio de Econom\'{\i}a y Competitividad grant MTM2015-63723-P}
\date{}
\begin{abstract}
As it is well-known, the forwards and pullback dynamics are in general unrelated. In this paper we present an in-depth study of whether the pullback attractor is also a forwards attractor for the processes involved with the skew-product semiflow induced by a family of scalar non-autonomous reaction-diffusion equations which are linear in a neighbourhood of zero and have null upper Lyapunov exponent. Besides, the notion of Li-Yorke chaotic pullback attractor for a process is introduced, and we prove that this chaotic behaviour might occur for almost all the processes. When the problems are additionally sublinear, more cases of forwards attraction are found, which had not been previously considered even in the case of linear-dissipative ODEs.
\end{abstract}
\renewcommand{\subjclassname}{\textup{2010} Mathematics Subject Classification}
\maketitle
\section{Introduction}\label{sec-intro}\noindent
In this work we continue the study of the dynamical  structure of the global and cocycle attractors of the skew-product semiflow  generated by a family of scalar non-autonomous linear-dissipative  reaction-diffusion equations over a minimal, uniquely ergodic and aperiodic flow $(P,\,{\cdot}\,,\R)$, with unique ergodic measure $\nu$. In particular this includes the important case of almost periodic differential equations. These linear-dissipative models contemplate the existence of two different zones:  first, a neighbourhood  around zero where the dissipation is negligible and the problems are linear; second, its complementary area, where a dissipative term, dominant with respect to the linear part, is added. In all what follows we assume that the upper Lyapunov exponent of the linear part of the equations is null. The reference Caraballo et al.~\cite{calaobsa} contains a detailed analysis of the structure of the global and cocycle attractors of these equations with Neumann or Robin boundary conditions, and now we extend the same conclusions to the case of Dirichlet boundary conditions.
\par
We consider standard regularity assumptions  which provide existence, uniqueness, existence in the large and continuous dependence of mild solutions with respect to initial data. Then, the mild solutions of the abstract Cauchy problems (ACPs for short) associated to the initial boundary value (IBV for short) problems  for the linear-dissipative equations generate a global skew-product semiflow $\tau$ on $P \times X$, where $X=C(\bar U)$ for Neumann and Robin boundary conditions and $X=C_0(\bar U)$ in the Dirichlet case.
The global dynamics of the semiflow $\tau$ is dissipative, but the dynamics of the linear semiflow $\tau_L$ induced by the linear part of the models has a strong influence on the structure of the global attractor $\A$  of $\tau$. We prove that also in the Dirichlet case $\A$ has lower and upper boundaries that can be identified with the graphs of two semicontinuous functions $a$ and $b$, respectively. For simplicity we assume that the coefficients of the equations are odd functions with respect to the state variable, which implies $a=-b$. We prove that generically the global attractor $\A=\cup_{p \in P}\{p\}\times A(p)$  is a pinched compact set with ingredients of dynamical complexity.
\par
A crucial fact in this work is that the linear semiflow $\tau_L$ is strongly monotone in $P \times X^\gamma$, where $X^\gamma=C({\bar U})$  for Neumann and Robin boundary conditions and $X^\gamma=X^\alpha$, a fractional power space associated with the realization of the Laplacian in an $L^p(U)$ space, in the case of Dirichlet boundary conditions.  In consequence, $\tau_L$ admits a continuous separation $X^\gamma=X_1(p)\oplus X_2(p)$ for  $p \in P$, in the terms stated in Pol\'{a}\v{c}ik and Tere\v{s}\v{c}\'{a}k~\cite{pote} and Shen and Yi~\cite{shyi}. The restriction of $\tau_L$ to the principal bundle $\cup_{p \in P}\{p\} \times X_1(p)$ defines a continuous 1-dimensional (1-dim for short) linear flow which provides the  upper Lyapunov exponent of the semiflow $\tau_L$. It also induces a 1-dim linear cocycle $c(t,p)$ which determines the behaviour  of the strongly positive trajectories of $\tau_L$ and the dynamical structure of the global attractor $\A$. It turns out that for each $p \in P$, the parameterized set $\{A(p{\cdot}t)\}_{t \in \R}$ is the pullback attractor of the process $S_p(\cdot,\cdot)$ generated  by the solutions of the non-autonomous parabolic equation obtained by the evaluation  of the coefficients along the trajectory of $p$. The main goal in this paper is to investigate  the forwards attraction properties (if any) of these pullback attractors. We will show how these properties are determined again by the behaviour of the 1-dim linear cocycle $c(t,p)$.
\par
We underline the fact that given a 1-dim linear cocycle $c(t,p)$ on $P$, there is a wide collection of scalar parabolic linear equations which provide $c$ (up to cocycle cohomology) on the principal bundle and thus, all the possible dynamical behaviours performed by the 1-dim linear cocycles on $P$ are achieved by appropriate scalar linear-dissipative reaction-diffusion equations.
\par
The conclusions of this paper extend the  results obtained in  Caraballo et al.~\cite{caloNonl}  for non-autonomous linear-dissipative scalar ODEs to the context of reaction-diffusion equations. Furthermore, we prove the forwards attraction of the pullback attractor in some dynamical situations where this matter had not been previously analyzed in the literature, and hence we also complete part of the results of the previous reference for ODEs.
\par
We next describe the structure and main results of the paper. Section~\ref{sec-preli} contains some basic facts in non-autonomous dynamical systems which will be required in the rest of the work.  In Section~\ref{sect-the problem} we show that the main dynamical ingredients used in Caraballo et al.~\cite{calaobsa} to describe the structure of the global attractor of the linear-dissipative reaction-diffusion equations with Neumann and Robin boundary conditions remain valid in the study of the attractor in the space $X^\alpha$ when the boundary conditions are of Dirichlet type. We also show that the upper Lyapunov exponents of the linear semiflows respectively on $P\times C_0(\bar U)$ and $P\times X^\alpha$ are the same, and that the restriction of both topologies  on the attractor $\A$ agree.
\par
We denote by $C_0(P \times \bar U)$ the set of continuous real maps $h$ on $P \times \bar U$, providing the linear coefficient of the problems, such that the upper Lyapunov exponent of the linear semiflow $\tau_L$ is null. $B(P \times \bar U)$ is the subset of the previous set of maps $h$ with an associated bounded cocycle $c(t,p)$, i.e., $\sup_{t \in \R,p\in P}|\ln c(t,p)|<\infty$; and $\mathcal{U}(P \times \bar U)= C_0(P \times \bar U)\setminus B(P \times \bar U)$. If $h \in B(P\times {\bar U})$, the simple structure of the sections $A(p)$ of $\A$ permits to apply the theory of Cheban et al.~\cite{chks} to conclude that the pullback attractor of all the processes is also a forwards attractor. The main objective of Section~\ref{sec-forwards attraction} is to  study  the same question when $h \in \mathcal{U}(P \times {\bar U})$. In this case there is a residual invariant subset of points where the attractor gets pinched, that is, a set $P_{\rm{s}} \subset P$ such that $b(p)=0$ for every $p \in P_{\rm{s}}$, and $P_{\rm{f}}=P\setminus P_{\rm{s}}$ is a dense invariant subset of first category and $b(p)\gg 0$ for every $p \in P_{\rm{f}}$. In addition,  $p \in P_{\rm{f}}$ if and only if $\sup_{t \leq 0} c(t,p)<\infty$. The set $P_{\rm{r}}$ contains the points $p \in P$ such that $p{\cdot}t$ is recurrent for every $t \in \R$. Other relevant sets are  $P_{\rm{o}}$, $P_{\rm{a}}^+$ and $P_{\rm{a}}^{-}$ which  contain respectively the oscillatory points, and the asymptotic points at $+\infty$ and  at $-\infty$. The classical references Poincar\'{e}~\cite{poin, poin86}  include precise examples of 1-dim smooth linear cocycles with recurrent, oscillatory or asymptotic points. All of them play an interesting role in the theory of non-autonomous strongly monotone linear parabolic PDEs.
\par
It is known that $P_{\rm{o}} \subset P_{\rm{r}}\cap  P_{\rm{s}}$, the sets $P_{\rm{o}}, P_{\rm{r}}$ are residual in $P$ with $\nu(P_{\rm{r}})=1$ and  $P_{\rm{a}}^{-} \subset P_{\rm{f}}$. We prove that if $p\in P_{\rm{s}}$, then the (null) pullback attractor of the process $S_p(\cdot,\cdot)$ is a forwards attractor if and only if $p\in P_{\rm{a}}^{+}$. This implies that for all the points in the residual set $P_{\rm{r}} \cap P_{\rm{s}}$ the pullback attractor is not forwards. Note  that  if $\nu(P_{\rm{s}})=1$, then also $\nu(P_{\rm{r}} \cap P_{\rm{s}})=1$. On the other hand, we prove that if $p \in P_{\rm{f}} \cap  P_{\rm{r}}$, the sections $A(p{\cdot}t)$, $t\in\R$ are included in the zone where the restriction of the reaction-diffusion equations is linear and the pullback attractor of $S_p(\cdot,\cdot)$ is also a forwards attractor. Moreover all the pairs of distinct points in $A(p)$  are pairs of Li-Yorke and hence we state that the pullback attractor $\{A(p{\cdot}t)\}_{t \in \R}$ is chaotic in the sense of Li-Yorke. Note  that if  $\nu(P_{\rm{f}})=1$, then also $\nu(P_{\rm{r}} \cap P_{\rm{f}})=1$. Last but not least, it is also proved in Section~\ref{sec-forwards attraction} that for any linear coefficient $h\in C_0(P \times \bar U)$ the sections of the attractor $A(p)\subset \Int X^\gamma_+\cup \Int X^\gamma_-\cup \{0\}$ for $p\in P$.
\par
In Section~\ref{sec-sublineal} we additionally assume that the dissipative term of the equations is sublinear.  Thus, our models include  linear-dissipative versions of the parabolic Fisher and Chafee-Infante equations. For $p \in P_{\rm{f}}$ we deduce that  all the pairs in the section $A(p)$  are strongly ordered, and if $p \in P_{\rm{f}}$ is such that  $\limsup_{t\to\infty} c(t,p)=\infty$, then for any $0<z<b(p)$ or  $z\gg 0$ with $b(p)\leq z$ (resp.~$-b(p)<z<0$ or $z\ll 0$ with $z\leq -b(p)$), the semitrajectory of $(p,z)$ approximates the upper (resp.~lower) boundary map  of the attractor as $t\to\infty$. From these results we conclude that for all such  $p \in P_{\rm{f}}$  the pullback attractor of the process $S_p(\cdot,\cdot)$ is also a forwards attractor.
\section{Basic notions}\label{sec-preli}\noindent
In this section we include some preliminaries about
topological dynamics for non-autonomous dynamical systems.
\par
Let $(P,d)$ be a compact
metric space. A real {\em continuous flow\/} $(P,\theta,\R)$ is
defined by a continuous map $\theta: \R\times P \to  P,\;
(t,p)\mapsto \theta(t,p)=\theta_t(p)=p{\cdot}t$ satisfying
\begin{enumerate}
\renewcommand{\labelenumi}{(\roman{enumi})}
\item $\theta_0=\text{Id},$
\item $\theta_{t+s}=\theta_t\circ\theta_s$ for each $s$, $t\in\R$\,.
\end{enumerate}
The set $\{ \theta_t(p) \mid t\in\R\}$ is called the {\em orbit\/}
of the point $p$. We say that a subset $P_1\subset P$ is {\em
$\theta$-invariant\/} if $\theta_t(P_1)=P_1$ for every $t\in\R$.
The flow $(P,\theta,\R)$ is called {\em minimal\/} if it does not contain properly any other
compact $\theta$-invariant set, or equivalently,  if every
orbit is dense. The flow is {\em distal\/} if the orbits of  any two distinct
points $p_1,\,p_2\in P$  keep at a positive distance,
that is, $\inf_{t\in \R}d(\theta(t,p_1),\theta(t,p_2))>0$; and it is {\em almost periodic\/} if the family of maps $\{\theta_t\}_{t\in \R}:P\to P$ is uniformly equicontinuous. An almost periodic flow is always distal.
\par
A finite regular measure defined on the Borel sets of $P$ is called
a Borel measure on $P$. Given $\mu$ a normalized Borel measure on
$P$, it is {\em $\theta$-invariant\/} if $\mu(\theta_t(P_1))=\mu(P_1)$ for every Borel subset
$P_1\subset P$ and every $t\in \R$. It is {\em ergodic\/}  if, in
addition, $\mu(P_1)=0$ or $\mu(P_1)=1$ for every
$\theta$-invariant Borel subset $P_1\subset P$.
$(P,\theta,\R)$ is {\em uniquely ergodic\/} if it has a
unique normalized invariant measure, which is then necessarily
ergodic. A minimal and almost periodic flow $(P,\theta,\R)$ is uniquely ergodic.
\par
A standard method to, roughly speaking, get rid of the time variation in a non-autonomous equation and build a non-autonomous dynamical system, is the so-called {\em hull\/} construction. More precisely, a function $f\in C(\R\times\R^m)$ is said to be {\em admissible\/} if for any
compact set $K\subset \R^m$, $f$ is bounded and uniformly continuous
on $\R\times K$. Provided that $f$ is admissible, its {\em hull\/} $P$ is the closure for the compact-open topology of the set of $t$-translates of $f$, $\{ f_t \mid t\in\R\}$ with $f_t(s,x)=f(t+s,x)$
for $s\in \R$ and $x\in\R^m$. The translation map $\R\times P\to P$,
$(t,p)\mapsto p{\cdot}t$ given by $p{\cdot}t(s,x)= p(s+t,x)$ ($s\in \R$ and $x\in\R^m$) defines a
continuous flow on the compact metric space $P$. This flow is minimal as far as the map $f$ has certain recurrent behaviour in time, such as periodicity, almost periodicity, or other weaker properties of recurrence. If the map $f(t,x)$ is uniformly almost periodic (that is, it is admissible and almost periodic in $t$ for any fixed $x$), then the flow on the hull is minimal and almost periodic, and thus uniquely ergodic. It is aperiodic whenever $f$ is not time-periodic. This is how an almost periodic equation is brought into the abstract context of this paper.
\par
Let $\R_+=\{t\in\R\,|\,t\geq 0\}$. Given a continuous compact flow $(P,\theta,\R)$ and a
complete metric space $(X,\dd)$, a continuous {\em skew-product semiflow\/} $(P\times
X,\tau,\,\R_+)$ on the product space $P\times X$ is determined by a continuous map
\begin{equation*}
 \begin{array}{cccl}
 \tau \colon  &\R_+\times P\times X& \longrightarrow & P\times X \\
& (t,p,x) & \mapsto &(p{\cdot}t,u(t,p,x))
\end{array}
\end{equation*}
 which preserves the flow on $P$, called the {\em base flow\/}.
 The semiflow property means:
\begin{enumerate}
\renewcommand{\labelenumi}{(\roman{enumi})}
\item $\tau_0=\text{Id},$
\item $\tau_{t+s}=\tau_t \circ \tau_s\;$ for all  $\; t$, $s\geq 0\,,$
\end{enumerate}
where again $\tau_t(p,x)=\tau(t,p,x)$ for each $(p,x) \in P\times X$ and $t\in \R_+$.
This leads to the so-called (nonlinear) semicocycle property,
\begin{equation*}
 u(t+s,p,x)=u(t,p{\cdot}s,u(s,p,x))\quad\mbox{for $s,t\ge 0$ and $(p,x)\in P\times X$}.
\end{equation*}
\par
The set $\{ \tau(t,p,x)\mid t\geq 0\}$ is the {\em semiorbit\/} of
the point $(p,x)$. A subset  $K$ of $P\times X$ is {\em positively
invariant\/} if $\tau_t(K)\subseteq K$
for all $t\geq 0$ and it is $\tau$-{\em invariant\/} if $\tau_t(K)= K$
for all $t\geq 0$.  A compact $\tau$-invariant set $K$ for the
semiflow  is {\em minimal\/} if it does not contain any nonempty
compact $\tau$-invariant set  other than itself.
\par
A compact $\tau$-invariant set $K\subset P\times X$ is called a {\em pinched\/} set if there exists a residual set $P_0\subsetneq P$ such that for every $p\in P_0$ there is a unique element in $K$ with $p$ in the first component, whereas there are more than one if $p\notin P_0$.
\par
The reader can find in  Ellis~\cite{elli}, Sacker and
Sell~\cite{sase94}, Shen and Yi~\cite{shyi} and references therein, a
more in-depth survey on topological dynamics.
\par
To finish, we include some basic notions on monotone skew-product semiflows. When the state space $X$ is a strongly ordered Banach space, that is, there is a closed convex cone of nonnegative vectors $X_+$ with a nonempty interior, then, a (partial) {\em strong order relation\/} on $X$ is
defined by
\begin{equation}\label{order}
\begin{split}
 x\le y \quad &\Longleftrightarrow \quad y-x\in X_+\,;\\
 x< y  \quad &\Longleftrightarrow \quad y-x\in X_+\;\text{ and }\;x\ne y\,;
\\  x\ll y \quad &\Longleftrightarrow \quad y-x\in \Int X_+\,.\qquad\quad\quad~
\end{split}
\end{equation}
In this situation, the skew-product semiflow $\tau$
is {\em monotone\/} if $u(t,p,x)\le u(t,p,y)$ for $t\ge 0$, $p\in P$ and  $x,y\in X$ with $x\le y$; and it is {\em strongly monotone\/} if besides, $u(t,p,x)\ll u(t,p,y)$ for $t>0$ and $p\in P$,  provided that $x< y$.
\par
A Borel map $a:P\to X$ such that $u(t,p,a(p))$ exists for any $t\geq 0$ is said to be
\begin{itemize}
\item[(a)] an {\em equilibrium\/} if $a(p{\cdot}t)=u(t,p,a(p))$ for any $p\in P$ and $t\geq 0\,$;
\item[(b)] a {\em sub-equilibrium\/} if $a(p{\cdot}t)\leq u(t,p,a(p))$ for any $p\in P$ and $t\geq 0\,$;
\item[(c)] a {\em super-equilibrium\/} if $a(p{\cdot}t)\geq u(t,p,a(p))$ for any $p\in P$ and $t\geq 0\,$.
\end{itemize}
The study of semicontinuity properties of these maps and other related issues can be found in Novo et al.~\cite{nono2}.
\section{Non-autonomous scalar linear-dissipative parabolic PDEs}\label{sect-the problem}\noindent
In this paper we consider a family of time-dependent scalar linear-dissipative parabolic PDEs over a minimal, uniquely ergodic and aperiodic flow $(P,\theta,\R)$ defined on a compact metric space $P$, with Neumann, Robin or Dirichlet boundary conditions, given for each $p\in P$ by
\begin{equation}\label{pdefamilynl}
\left\{\begin{array}{l} \des\frac{\partial y}{\partial t}  =
 \Delta \, y+h(p{\cdot}t,x)\,y+g(p{\cdot}t,x,y)\,,\quad t>0\,,\;\,x\in U,
  \\[.2cm]
By:=\alpha(x)\,y+\delta\,\des\frac{\partial y}{\partial n} =0\,,\quad  t>0\,,\;\,x\in \partial U,\,
\end{array}\right.
\end{equation}
where  $p{\cdot}t$ denotes the flow on $P$; $U$, the spatial domain, is a bounded, open and
connected  subset of $\R^m$ ($m\geq 1$) with a sufficiently smooth boundary
$\partial U$; $\Delta$ is the Laplacian operator on $\R^m$; the linear coefficient $h:P\times \bar U\to \R$ is continuous
and the nonlinear term $g:P\times \bar U\times \R\to \R$ is continuous and of class $C^1$ with respect to  $y$ and satisfies the following conditions which in particular render the equations dissipative and delimit a linear zone for the problems:
\begin{itemize}
  \item[(c1)] $g(p,x,0)=\des\frac{\partial g}{\partial y}(p,x,0)=0$ for any $p\in  P$ and $x\in \bar U$;
  \item[(c2)] $y\,g(p,x,y)\leq 0$ for any $p\in  P$, $x\in \bar U$ and $y\in \R$;
  \item[(c3)] $\des\lim_{|y|\to\infty}\frac{g(p,x,y)}{y}=-\infty$ uniformly on $P\times\bar U$;
  \item[(c4)] $g(p,x,-y)=-g(p,x,y)$ for any $p\in  P$, $x\in \bar U$ and $y\in \R$;
  \item[(c5)] there exists an $r_0 > 0$ such that $g(p,x,y)=0$ if and only if $|y|\leq r_0$.
\end{itemize}
\par
The problem has Dirichlet boundary conditions if $\delta=0$ and $\alpha(x)\equiv 1$; Neumann boundary conditions if $\delta=1$ and  $\alpha(x) \equiv 0$; and Robin boundary conditions if  $\delta=1$ and $\alpha:\partial U\to \R$ is a nonnegative  sufficiently regular map. Recall that
$\partial/\partial n$ denotes the
outward normal
derivative at the  boundary.
\par
Section 3.1 in Cardoso et al.~\cite{cardoso} is devoted to the existence of attractors for linear-dissipative parabolic PDEs of type~\eqref{pdefamilynl} with conditions (c1), (c2) and (c3) on the nonlinear term $g$, and Neumann or Robin boundary conditions. Thanks to the presence of the dissipative term $g(p,x,y)$, there exists an absorbing compact set for the skew-product semiflow induced by the mild solutions of the associated ACPs in $X=C(\bar U)$,
\begin{equation}\label{tau}
\begin{array}{cccl}
 \tau: & \R_+\times P\times X& \longrightarrow & \hspace{0.3cm}P\times X\\
 & (t,p,z) & \mapsto
 &(p{\cdot}t,u(t,p,z))\,,
\end{array}
\end{equation}
and thus  $\tau$ has a global attractor $\A\subset P\times X$, which is formed by bounded entire trajectories. More precisely, $\A$ is an invariant compact set attracting bounded sets forwards in time, i.e., $\lim_{t\to\infty} {\rm dist}(\tau_t(\B),\A)=0$ for any bounded set $\B\subset P\times X$, for the Hausdorff semidistance (see also Caraballo et al.~\cite{calaobsa} for all the details).
\par
The semiflow $\tau$ is globally defined because of the boundedness of the solutions, it is strongly monotone, and the section semiflow $\tau_t$ is compact for every $t>0$ (see Travis and Webb~\cite{trwe}).
Besides, in~\cite{cardoso} the structure of the attractor is studied in  the cases $\lambda_P<0$ and $\lambda_P>0$, where $\lambda_P$ is  the upper Lyapunov exponent  of the linearized family along the null solution,
\begin{equation}\label{linealizada}
\left\{\begin{array}{l} \des\frac{\partial y}{\partial t}  =
 \Delta \, y+h(p{\cdot}t,x)\,y\,,\quad t>0\,,\;\,x\in U, \;\, \text{for each}\; p\in P,\\[.2cm]
By:=\alpha(x)\,y+\delta\,\des\frac{\partial y}{\partial n} =0\,,\quad  t>0\,,\;\,x\in \partial U,
\end{array}\right.
\end{equation}
whereas Caraballo et al.~\cite{calaobsa} have studied the structure of the attractor when $\lambda_P=0$. Still many interesting problems remain open in the difficult case $\lambda_P=0$. We will address some of them in the forthcoming sections.
\par
Hereafter, we denote by $\tau_L$ the  linear skew-product semiflow
\begin{equation}\label{tauL}
\begin{array}{cccl}
 \tau_L: & \R_+\times P\times X& \longrightarrow & \hspace{0.3cm}P\times X\\
 & (t,p,z) & \mapsto
 &(p{\cdot}t,\phi(t,p)\,z)\,
\end{array}
\end{equation}
induced by  the mild solutions $\phi(t,p)\,z$ of the linear ACPs associated to~\eqref{linealizada}. In particular $\phi(t,p)$ are bounded linear  operators on $X$ which are compact for $t>0$  and satisfy the linear semicocycle property $\phi(t+s,p)=\phi(t,p{\cdot}s)\,\phi(s,p)$, $t,s\geq 0$, $p\in P$. A crucial property is that these operators are also strongly positive, i.e., for $p\in P$ and $t>0$, $\phi(t,p)\,z\gg 0$ if $z>0$ (recall that $\Int X_+=\{z\in X\mid z(x)> 0 \;\forall x\in \bar U\}$).  The reason is that the operators $\phi(t,p)$ being compact and strongly positive make the linear skew-product semiflow $\tau_L$ admit a continuous separation (see Pol\'{a}\v{c}ik and Tere\v{s}\v{c}\'{a}k~\cite{pote} in the discrete case, and Shen and Yi~\cite{shyi} in the continuous case). This means that there are two families of subspaces $\{X_1(p)\}_{p\in P}$ and $\{X_2(p)\}_{p\in P}$ of $X$ which satisfy:
\begin{itemize}
\item[(1)] $X=X_1(p)\oplus X_2(p)$  and $X_1(p)$, $X_2(p)$ vary
    continuously in $P$;
 \item[(2)] $X_1(p)=\langle e(p)\rangle$, with $e(p)\gg 0$ and
     $\|e(p)\|=1$ for any $p\in P$;
\item[(3)] $X_2(p)\cap X_+=\{0\}$ for any $p\in P$;
\item[(4)] for any $t>0$,  $p\in P$,
\begin{align*}
\phi(t,p)\,X_1(p)&= X_1(p{\cdot}t)\,,\\
\phi(t,p)\,X_2(p)&\subset X_2(p{\cdot}t)\,;
\end{align*}
\item[(5)] there are $M>0$, $\delta>0$ such that for any $p\in P$, $z\in
    X_2(p)$ with $\|z\|=1$ and $t>0$, $\|\phi(t,p)\,z\|\leq M \,e^{-\delta t}\|\phi(t,p)\,e(p)\|$.
\end{itemize}
\par
In this situation, the 1-dim invariant subbundle
\begin{equation}\label{principal bundle}
\displaystyle\bigcup_{p\in P} \{p\} \times X_1(p)\,
\end{equation}
is called the {\em principal bundle\/} and the Sacker-Sell spectrum of the restriction of $\tau_L$ to this invariant subbundle is called the {\em principal spectrum\/} of $\tau_L$, and is denoted by  $\Sigma_{\text{pr}}$ (see Mierczy{\'n}ski and Shen~\cite{mish}). In the uniquely ergodic case, $\Sigma_{\text{pr}}=\{\lambda_P\}$.
\par
Besides, the continuous separation permits to  associate to $h$ a 1-dim continuous linear cocycle  $c(t,p)$, given for $t\geq 0$ by the positive numbers such that
\begin{equation}\label{c}
\phi(t,p)\,e(p)=c(t,p)\,e(p{\cdot}t)\,,\quad t\geq 0\,,\; p\in P
\end{equation}
and by the relation $c(-t,p)=1/c(t,p{\cdot}(-t))$ for any $t>0$ and $p\in P$. This 1-dim linear cocycle is a fundamental tool in our study. Roughly speaking, it allows us to apply techniques of the scalar world to our current infinite dimensional problems.
\begin{nota}\label{nota-cohomologos}
One is usually familiar with 1-dim smooth linear cocycles $c_0(t,p)$, which are those for which the map $a(p):=\left.\frac{d}{dt} \ln c_0(t,p)\right|_{t=0}$ exists and is continuous on $P$, so that $c_0(t,p)$ is just the fundamental solution $c_0(t,p)=\exp\int_0^t a(p{\cdot}s)\,ds$ of the scalar linear ODEs $y'=a(p{\cdot}t)\,y$ ($p\in P$). It is important to note that, as it is proved in Lemma~3.2 in Johnson et al.~\cite{jops}, any continuous cocycle $c(t,p)$ is cohomologous to a smooth one, that is, there exists a smooth cocycle $c_0(t,p)=\exp\int_0^t a(p{\cdot}s)\,ds$ for some $a\in C(P)$, and a continuous map $f:P\to \R\setminus\{0\}$ such that $c(t,p)= f(p{\cdot}t)\,c_0(t,p)\,f(p)^{-1}$ for $p\in P$ and $t\in\R$.
\par
Now, when we consider the problems for a given $h\in C(P\times \bar U)$ with associated continuous cocycle $c(t,p)$, and for a $k\in C(P)$,
\begin{equation}\label{pdefamily h+k}
\left\{\begin{array}{l} \des\frac{\partial y}{\partial t}  =
 \Delta \, y+h(p{\cdot}t,x)\,y+k(p{\cdot}t)\,y\,,\quad t>0\,,\;\,x\in U, \;\, \text{for each}\; p\in P,\\[.2cm]
By:=\alpha(x)\,y+\delta\,\des\frac{\partial y}{\partial n} =0\,,\quad  t>0\,,\;\,x\in \partial U,\,
\end{array}\right.
\end{equation}
the associated cocycle is $\wit c(t,p)= c(t,p)\,\exp \int_0^t k(p{\cdot}s)\,ds$ (see~\cite{calaobsa}), which is trivially cohomologous to the smooth cocycle $\exp\int_0^t (a+k)(p{\cdot}s)\,ds$, with the same cohomology map $f$ independently of $k$. Varying the map $k\in C(P)$, we get a cocycle for~\eqref{pdefamily h+k} which is cohomologous to any {\em a priori\/} fixed smooth cocycle.
\end{nota}
\subsection{The case of Dirichlet boundary conditions}\label{subsect-Dirichlet}
One important issue is to extend our results to the case of Dirichlet boundary conditions. Some extensions are for free, once we are in the appropriate context. To establish the appropriate context is somewhat lengthy. The reason is that in order to have strong monotonicity and a continous separation for the linear semiflow $\tau_L$ as before, we need a strong order in the Banach space where one defines the associated ACPs. Since the most common choice $C_0(\bar U)$ of the continuous maps vanishing on the boundary $\partial U$ has a positive cone with an empty interior, we will have to resort to an intermediate space, or more precisely, to a domain of fractional powers associated to the realization of the Dirichlet Laplacian in an $L^p(U)$ space.
\par
Although a more general family of problems under fewer regularity assumptions could be considered, in this section we restrict attention to the family~\eqref{pdefamilynl} with the conditions there exposed and with Dirichlet boundary conditions, that is, $y(t,x)=0$ for $t>0$ and $x\in \partial U$. This means that initial value maps $z\in C(\bar U)$ must satisfy the restriction $z\in C_0(\bar U)$. So, for Dirichlet boundary conditions we start with the Banach space  $X=C_0(\bar U)$ with the sup-norm $\|\,{\cdot}\,\|$, trying to keep a common notation for all the types of boundary conditions, as far as possible.
\par
Now, following for instance Smith~\cite{smit}, we take  $A$  the closure of the differential operator
$A_0\colon  D(A_0)\subset X\to X$,  $A_0z=\Delta z$, defined on $D(A_0)=\{z\in C^2(U)\cap C_0(\bar U)\;|\; A_0z\in C_0(\bar U)\}$.
The operator $A$ is sectorial and it generates
an analytic compact semigroup of operators $\{T(t)\}_{t\geq 0}$ on $X$
 which is strongly continuous (that is, $A$ is densely defined).
Since $g(p,x,0)=0$ for any $p\in P$ and $x\in \bar U$, the map
$\tilde f:P\times X\to X$, $(p,z)\mapsto \tilde f(p,z)$, $\tilde f(p,z)(x)=h(p,x)\,z(x)+g(p,x,z(x))$,  $x\in \bar U$ is well defined.
Arguing exactly as in Caraballo et al.~\cite{calaobsa} (where only Neumann or Robin boundary conditions are considered), we build the family  $(p\in P)$ of associated ACPs in the space $X$,
\begin{equation*}
\left\{\begin{array}{l} u'(t)  =
 A\, u(t)+\tilde f(p{\cdot}t,u(t))\,,\quad t>0\,,\\
u(0)=z\,,
\end{array}\right.
\end{equation*}
which admit unique mild solutions $u(t,p,z)$, that is, continuous maps defined on  maximal intervals $[0,\beta)$ for some $\beta=\beta(p,z)>0$ (possibly $\infty$) which satisfy the integral equations $(p\in P)$
\begin{equation}\label{ec integral}
 u(t)=T(t)\,z +\int_0^t T(t-s)\,
 \tilde f(p{\cdot}s,u(s))\,ds\,,\quad t\in [0,\beta)\,.
\end{equation}
Mild solutions permit to define (in principle, only locally) a continuous
skew-product semiflow $\tau$~\eqref{tau}. Besides, for any $t>0$ the section map $\tau_t$ is compact, meaning that it takes bounded sets in $P\times X$ into relatively compact sets (see Proposition~2.4 in~\cite{trwe}, where the compactness of the operators $T(t)$ for $t>0$ is crucial); and if a solution $u(t,p,z)$ remains bounded, then it is defined on the whole positive real line and the semiorbit of $(p,z)$ is relatively compact. Note that the presence of the dissipative term $g$ makes all solutions bounded, so that the semiflow $\tau$ is globally defined. With respect to the linear problems~\eqref{linealizada}, the same treatment leads to the linear skew-product semiflow $\tau_L$~\eqref{tauL}, which this time is globally defined thanks to linearity.
Besides, we have the expected monotonicity of the semiflows.
\begin{prop}\label{prop-monot Dirichlet}
The skew-product semiflows $\tau$ and $\tau_L$ on $P\times X$, induced by mild solutions, are monotone.
\end{prop}
\begin{proof}
We do not give the proof in detail, since it suffices to follow the ideas in the proof of Theorem~3.1 in~\cite{calaobsa}. Basically, one approximates the initial problem by a sequence of regular problems whose solutions  are classical, so that the standard methods of comparison of solutions (for instance, see Fife and Tang~\cite{fita}) apply to these problems, and the obtained inequalities are preserved when taking limits.
 \end{proof}
At this point we introduce the strongly ordered Banach space $X^{\alpha}$ with norm $\|{\cdot}\|_\alpha$ defined as follows.
Let us consider the realization of the Dirichlet Laplacian on the Banach space $L^p(U)$ for a fixed $m<p<\infty$, that is, the operator $A_{p}:D(A_{p})\subset L^p(U)\to L^p(U)$ 
defined by $A_{p}z=\Delta z$ (in a weak sense) for $z\in D(A_{p})$. This operator is sectorial, densely defined and $0\in\rho(A_{p})$. Then, for $\alpha\in (1/2+m/(2p),1)$, let $X^\alpha:=D(-A_{p})^\alpha$ be the domain of fractional power $\alpha$ of $-A_{p}$, which is a Banach space with norm
$\|z\|_\alpha = \|(-A_{p})^{\alpha}\,z\|_p$
and satisfies $X^\alpha\hookrightarrow C^1(\bar U)\cap C_0(\bar U)$ (see Theorem~1.6.1 in Henry~\cite{henr}).
As it is standard, the (partial) strong order in the Banach space $X^\alpha$ is defined as in~\eqref{order} in  association with the cone of positive maps $X^\alpha_+=\{z\in X^\alpha\,\big|\; z(x)\geq 0\;\text{for} \;
x\in\bar U \}$, which has a nonempty interior: 
\begin{equation*}
 \Int X^\alpha_+ =\Big\{z\in X^\alpha_+\,\big|\; z(x)> 0\;\text{for} \;
x\in U \;\text{and}\; \frac{\partial z}{\partial n}(x)< 0 \;\text{for} \;
x\in\partial U   \Big\}\,.
\end{equation*}
In fact $X^{\alpha}$ is an intermediate space between $X$ and $D(A)$ and the semigroup of operators $\{T(t)\}_{t\geq 0}$ is a strongly continuous analytic compact semigroup on $X^{\alpha}$. In particular, for $t>0$, $T(t):X\to X^{\alpha}$ is a bounded operator.
\par
The theory for semilinear ACPs with nonlinearities defined in intermediate spaces (for instance, see Chapter~7 in Lunardi~\cite{luna}) asserts that the associated ACPs
\begin{equation*}
\left\{\begin{array}{l} u'(t)  =
 A\, u(t)+\tilde f(p{\cdot}t,u(t))\,,\quad t>0\,,\\
u(0)=z\in X^{\alpha}\,
\end{array}\right.
\end{equation*}
 admit mild solutions, where $\tilde f:P\times X^\alpha\to X$, $(p,z)\mapsto \tilde f(p,z)$, $\tilde f(p,z)(x)=h(p,x)\,z(x)+g(p,x,z(x))$,  $x\in \bar U$. In this case mild solutions $u(t)=u(t,p,z)$ are maps in $C([0,\beta),X^{\alpha})$ defined on  maximal intervals for some $\beta=\beta(p,z)>0$ (possibly $\infty$)  which satisfy the integral equations~\eqref{ec integral}, just the same as before.
Then, mild solutions permit us to define a continuous skew-product semiflow:
\begin{equation*}
\begin{array}{cccl}
 \tau: & \R_+\times P\times X^{\alpha}& \longrightarrow & \hspace{0.3cm}P\times X^{\alpha}\\
 & (t,p,z) & \mapsto
 &(p{\cdot}t,u(t,p,z))\,.
\end{array}
\end{equation*}
In particular, since  for $t>0$,  $T(t)$ maps $X$ into $X^{\alpha}$, it is not difficult to deduce from~\eqref{ec integral} that for $z\in X$, $u(t,p,z)\in X^{\alpha}$ for $t>0$: in order to control the $\alpha$-norm of the integrand, apply  that $\sup_{0<s\leq t}\|s^{\alpha}\,T(s)\|_{\mathcal{L}(X,X^{\alpha})}<\infty$ for any $t>0$ (for instance, see~\cite{luna}) and use that the semiorbit $u(t,p,z)\in X$ is bounded, so that so is the sup-norm of $\tilde f$. Actually,
once more using arguments from Travis and Webb~\cite{trwe}, one can prove that the section semiflow $\tau_t:P\times X\to P\times X^\alpha$ is compact provided that $t>0$. Then, arguing exactly as in Proposition~2 in Obaya and Sanz~\cite{obsa2019}, we can state the following result.
\begin{prop}\label{prop-same topology}
If $K$ is a compact $\tau$-invariant subset of $P\times X$, then $K \subset P \times X^\alpha$ and the restriction of both topologies on $K$ agree.
\end{prop}
In order to get the existence of the global attractor for $\tau$, we prove this result.
\begin{prop}\label{prop-absorbente}
Consider the family of problems~\eqref{pdefamilynl} with Dirichlet boundary conditions, with $h$ and $g$ as there described, and assume conditions $\rm{(c1)}$-$\rm{(c3)}$ on $g$. Then, there is a compact set $C_1\subset P\times X$ which is absorbing, that is, for any bounded set $B\subset X$ there exists a $t_0=t_0(B)$ such that $\tau_t(P\times B)\subset  C_1$ for $t\geq t_0$.
\end{prop}
\begin{proof}
We first assume enough regularity on the maps $h(p{\cdot}t,x)$ and $g(p{\cdot}t,x,y)$ with respect to $t$ and $x$ (a H\"{o}lder-continuity condition is enough: see Freedman~\cite{frie}) in order to have classical solutions out of mild solutions, that is, fixed $p\in P$ and  $z\in X$, $y(t,x)=u(t,p,z)(x)$ is a classical solution of the IBV problem given by~\eqref{pdefamilynl} plus the initial condition $y(0,x)=z(x)$, $x\in \bar U$.
\par
With conditions (c1)-(c3) we can choose an $r^*>0$ such that, writing $G(p,x,y)=h(p,x)\,y+g(p,x,y)$,
\begin{equation}\label{r^*}
\begin{split}
&G(p,x,y)<0 \quad\text{for}\;\, p\in P,\,\; x\in \bar U,\,\; y\geq r^*, \\
&G(p,x,y)>0 \quad\text{for}\;\, p\in P,\,\; x\in \bar U,\,\; y\leq -r^*.
\end{split}
\end{equation}
Let $B_{r^*}=\{z\in X\mid \|z\|\leq r^*\}$ and $C_1=\cls\{\tau(1,p,z)\mid p\in P,\, z\in B_{r^*}\}$. Since the section map $\tau_1$ is compact, the set $C_1$ is compact. To see that it is also absorbing, it suffices to see that $P\times B_{r^*}$ is absorbing.  For this purpose, let $\gamma_0>0$ be the  first eigenvalue of the  boundary value problem
\begin{equation}\label{bvp}
\left\{\begin{array}{l}
 \Delta \, u +\lambda\,u = 0\,,\quad x\in U,\\
u(x)=0\,,\quad x\in \partial U,
\end{array}\right.
\end{equation}
with associated eigenfunction $e_0\gg 0$ with $\|e_0\|=1$. For each $p\in P$ and $r\geq r^*$, let us denote  $y(t,x)=u(t,p,re_0)(x)\geq 0$, for $t\geq 0$, $x\in\bar U$. Now, for any $T>0$, let $m_T=\max\{\|u(t,p,re_0)\|\mid t\in [0,T]\}=\max\{y(t,x)\mid t\in [0,T], \,x\in\bar U\}$. Clearly, $m_T\geq \|re_0\|=r$. If the value $m_T$ is attained by $y(t,x)$ at a pair $(t_0,x_0)$, necessarily $x_0\in U$, since $y(t,x)$ vanishes if $x\in\partial U$. If it were $t_0\in (0,T)$, then it would be a local maximum for $y$, so that $\frac{\partial y}{\partial t}(t_0,x_0)=0$ and the hessian matrix of $u(t_0,p,re_0)$ at $x_0$ is negative-semidefinite, so that in particular its trace $\Delta y(t_0,x_0)\leq 0$. Since $y(t_0,x_0)\geq r$ and then $G(p{\cdot}t_0,x_0,y(t_0,x_0))<0$, the relation $\frac{\partial y}{\partial t}(t_0,x_0)=\Delta y(t_0,x_0)+G(p{\cdot}t_0,x_0,y(t_0,x_0))$ cannot hold. If it were $(t_0,x_0)=(T,x_0)$ with $x_0\in U$, this time $\frac{\partial y}{\partial t}(T,x_0)\geq 0$ and $\Delta y(T,x_0)\leq 0$, so that once more  the relation $\frac{\partial y}{\partial t}(T,x_0)=\Delta y(T,x_0)+G(p{\cdot}T,x_0,y(t_0,x_0))$ cannot hold. Therefore, the maximum is only attained at $t_0=0$, that is, $m_T=r$ and  $\|u(t,p,re_0)\|<r$ for any $t\in (0,T]$. Thus, $m(t)=\|u(t,p,re_0)\|$ is decreasing on a nontrivial interval to the right of $0$.
\par
At this point, let $t_1=\sup\{t_2\geq 0\mid \|u(t,p,re_0)\|\geq r^* \;\forall\,t\in [0,t_2) \}$. If $t_1<\infty$, then $\|u(t_1,p,re_0)\|= r^*$ and arguing just as before we get that for any $T>0$, $\max\{\|u(t,p{\cdot}t_1,u(t_1,p,re_0))\|\mid t\in [0,T]\}=r^*$. That is to say that for any $t\geq t_1$, $u(t,p,re_0)\in B_{r^*}$ and the orbit of $(p,re_0)$ is absorbed by the set $P\times B_{r^*}$.
\par
Let us see that it cannot be $t_1=\infty$. For if it were so, then $m(t)=\|u(t,p,re_0)\|$ would be always decreasing, once more arguing as before. Now, since the semiorbit $\{(p{\cdot}t,u(t,p,re_0))\mid t\geq 0\}$ is relatively compact, there exists a sequence $(t_n)_n\uparrow \infty$ such that $(p{\cdot}t_n,u(t_n,p,re_0))\to (p_1,z_1)\in P\times X_+$. Then, $\lim_{n\to\infty}\|u(t_n,p,re_0)\|=\|z_1\|\geq r^*$ and the map $\|u(t,p_1,z_1)\|$ is decreasing on an interval $[0,\delta]$ for a $\delta>0$, and in particular $\|z_1\|>\|u(\delta,p_1,z_1)\|$. Note that we can assume without loss of generality that $t_{n+1}-t_n\geq \delta$ for any $n\geq 1$ (otherwise, just take a subsequence). But then, by the decreasing character of $m(t)$, $\|u(t_{n+1},p,re_0)\|<\|u(t_{n}+\delta,p,re_0)\|$ for all $n\geq 1$, and taking limits, $\|z_1\|\leq \|u(\delta,p_1,z_1)\|$, which is a contradiction.
\par
Once we have the absorption for pairs $(p,re_0)$ with $r\geq r^*$, we get the absorption for $(p,-re_0)$ with $r\geq r^*$ in a similar way. Besides,  it is not hard to prove that there exists a $t_0>0$ such that $u(t,p,\pm re_0)\in B_{r^*}$ for any $t\geq t_0$ and $p\in P$. Then, given a bounded set $B\subset X$, $\tau_1(P\times B)$ is a relatively compact set of $P\times X^\alpha$,  and since $e_0\gg 0$ we can take an $r\geq r^*$ so that $-re_0\leq u(1,p,z)\leq re_0$ for $p\in P$ and $z\in B$. By monotonicity, $u(t,p{\cdot}1,-re_0)\leq u(t+1,p,z)\leq u(t,p{\cdot}1,re_0)$ for $t\geq 0$, and thus for $t\geq t_0+1$,  $u(t,p,z)\in B_{r^*}$ for any $p\in P$ and $z\in B$.
\par
To finish, we have to remove the supplementary regularity assumptions we have made. In the general case we can build maps $\wit h:P\times \bar U\to \R$ continuous and with the previous regularity assumptions  and $\wit g: \R\to \R$ of class $C^1$  plus conditions $\rm{(c1)}$-$\rm{(c3)}$, and such that $h\leq \wit h$ and $g(p,x,y)\leq \wit g(y)$ if $y\geq 0$ and $\wit g(y)\leq g(p,x,y)$ if $y\leq 0$, and consider the semiflow $\wit\tau(t,p,z)=(p{\cdot}t,\wit u(t,p,z))$ induced by the solutions of the problems~\eqref{pdefamilynl} with coefficients $\wit h, \wit g$. Then, arguing as in Theorem~3.1 in~\cite{calaobsa} with Dirichlet boundary conditions, we can compare solutions of the two problems to get  $0\leq u(t,p,re_0)\leq \wit u(t,p,re_0)$ and $\wit u(t,p,-re_0)\leq u(t,p,-re_0)\leq 0$ for $t\geq 0$. Since for $\wit \tau$ there is an $r^*>0$ such that for $r\geq r^*$, there is a $t_0>0$ such that $\wit u(t,p,\pm re_0)\in B_{r^*}$ for any $t\geq t_0$ and $p\in P$,  the proof is finished as before.
\end{proof}
As a corollary (see Kloeden and Rasmussen~\cite{klra}), we get the existence of a global attractor $\A\subset P\times X$ for the skew-product semiflow $\tau$ in the case of Dirichlet  boundary conditions. By Proposition~\ref{prop-same topology}, $\A\subset P\times X^\alpha$ and the restriction of both topologies on $\A$ agree.
Besides,  since $P$ is compact,  the non-autonomous set $\{A(p)\}_{p\in P}$, with $A(p)=\{z\in X\mid (p,z)\in \A\}$ for each $p\in P$, is a cocycle attractor (or a pullback attractor). This means that   $\{A(p)\}_{p\in P}$ is compact, invariant and it pullback attracts all bounded subsets $B\subset X$, that is,
\begin{equation}\label{pullback}
\lim_{t\to\infty} {\rm dist}(u(t,p{\cdot}(-t),B),A(p))=0\quad \text{for any}\; p\in P.
\end{equation}
\par
The next result is the counterpart of Proposition~3 in Cardoso et al.~\cite{cardoso}, now for Dirichlet boundary conditions. A similar result holds for $a(p)=\inf A(p)$.
\begin{prop}\label{prop-b Dirichlet}
Let $e_0\gg 0$ be the eigenfunction associated to the first eigenvalue of the boundary value problem~\eqref{bvp} with $\|e_0\|=1$. Then, for $r>0$ and $t_0>0$ both big enough, $b(p):=\lim_{n\to\infty} u(nt_0,p{\cdot}(-nt_0),re_0)$ is well defined in $X^\alpha$. Besides, $b(p)=\sup A(p)$ and  $b:P\to X^\alpha$ is a semicontinuous equilibrium for $\tau$.
\end{prop}
\begin{proof}
Since $\tau_1:P\times X\to P\times X^\alpha$ is compact, we can take  the compact absorbing set $C_1=\cls \tau_1(P\times B_{r^*})\subset P\times X^\alpha$ (see~\eqref{r^*} for the choice of $r^*$), and a big $r_0>0$   so that $-r_0e_0\leq z\leq r_0e_0$ whenever $(p,z)\in C_1$. Let us fix an $r\geq r_0$ and take a $t_0>0$ such that $\tau(t,p, re_0)\in C_1$ for any $t\geq t_0$ and $p\in P$.  Now, take any sequence $(t_n)_n\uparrow \infty$ with $t_1\geq t_0$. Since $\tau(t_n,p{\cdot}(-t_n),re_0)\in C_1$ for any $n\geq 1$, we can assume without loss of generality that there exists $b(p)=\lim_{n\to\infty}u(t_n,p{\cdot}(-t_n),re_0)$. By the pullback attraction of the section, see~\eqref{pullback}, $b(p)\in A(p)$. To see that $b(p)=\sup A(p)$, take any $z\in A(p)$ and let us check that $z\leq b(p)$. Since the orbits in $\A$ are full orbits, for each $n\geq 1$ we can take $(p{\cdot}(-t_n),z_n)\in \A\subseteq C_1$ such that $u(t_n,p{\cdot}(-t_n),z_n)=z$. Then, $z_n\leq r_0e_0\leq re_0$ for $n\geq 1$, so that by monotonicity, $z=u(t_n,p{\cdot}(-t_n),z_n)\leq u(t_n,p{\cdot}(-t_n),re_0)$, and taking limits $z\leq b(p)$. Note that in particular the value of the limit $\lim_{n\to\infty}u(t_n,p{\cdot}(-t_n),re_0)$ (taking a subsequence if necessary) is independent of the sequence $(t_n)_n\uparrow \infty$ considered.
\par
It is easy to check that $b$ defines an equilibrium for $\tau$: just take $p\in P$, $t\geq 0$ and a sequence $(t_n)_n\uparrow \infty$ such that there exist $\lim_{n\to\infty} u(t_n-t,p{\cdot}(t-t_n),re_0)=b(p)$ and $\lim_{n\to\infty} u(t_n,p{\cdot}(t-t_n),re_0)=b(p{\cdot}t)$. By the semicocycle property and the continuity of $u(t,p,\,{\cdot}\,)$, $u(t_n,p{\cdot}(t-t_n),re_0)=u(t,p,u(t_n-t,p{\cdot}(t-t_n),re_0))\to u(t,p,b(p))$ as $n\to\infty$ and thus, $b(p{\cdot}t)=u(t,p,b(p))$, as wanted.
\par
Now, consider the map $\tau_{t_0}:P\times X^\alpha\to P\times X^\alpha$ given by the semiflow at time $t_0$, and consider the discrete skew-product semiflow obtained by its iteration. Then,  the constant map $P\to X^\alpha$, $p\mapsto re_0$ is a continuous super-equilibrium for this discrete semiflow, that is, it satisfies that
$u(nt_0,p,re_0)\leq re_0$ for $p\in P$ and $n\geq 1$, because $\tau(nt_0,p,re_0)\in C_1$ for $p\in P$ and  $n\geq 1$. Then, we can use the method described in the proof of Theorem~3.6 in Novo et al.~\cite{nono2} (see also Chueshov~\cite{chue}) to build, starting from the continuous super-equilibrium $re_0$, a decreasing family of continuous super-equilibria $(b_n)_{n\geq 1}:P\to X^\alpha$ given precisely by $b_n(p)=u(nt_0,p{\cdot}(-nt_0),re_0)$ and  $\lim_{n\to\infty}b_n(p)=\inf_{n\geq 1}b_n(p)$ exists and defines a semicontinuous map. To conclude the proof, just note that the limit coincides with $b(p)$.
\end{proof}
Note that  the linear skew-product semiflow $\tau_L:\R_+\times P\times X^\alpha\to P\times X^\alpha$, $(t,p,z)\mapsto (p{\cdot}t,\phi(t,p)\,z)$ is built just as before, and $\phi(t,p)\in \mathcal{L}(X^\alpha)$ are compact for $t>0$.  Actually, $\phi(t,p):X\to X^\alpha$ is compact for $t>0$. We need the strong positivity of these operators in order to have a continuous separation for $\tau_L$.
\begin{prop}\label{prop-strongly positive}
The operators $\phi(t,p): X\to X^\alpha$ are strongly positive for $p\in P$ and $t>0$, that is, if $z>0$, then $\phi(t,p)\,z\gg 0$.
\end{prop}
\begin{proof}
Let us fix a $p\in P$, and take a $z\in X$ with $z>0$.   Note that $v(t)=\phi(t,p)\,z$ is the solution of  the integral equation
\begin{equation*}
 v(t)=T(t)\,z +\int_0^t T(t-s)\,
 \tilde h(p{\cdot}s)\,v(s)\,ds\,,\quad t\geq 0\,,
\end{equation*}
for the continuous and bounded map $\tilde h:P\to X$, $p\mapsto \tilde h(p)$, $\tilde h(p)(x)=h(p,x)$,  $x\in \bar U$. We already know that $v(t)\geq 0$ for $t\geq 0$, by Proposition~\ref{prop-monot Dirichlet}.
Also, the standard parabolic maximum principle implies that $T(t)$ is strongly positive for $t>0$; more precisely, $T(t)\,(X_+\setminus\{0\})\subset \Int X^{\alpha}_+$ (for instance, see Smith~\cite{smit}). Therefore, $T(t)\,z \gg 0$ for $t>0$.
\par
Now, if it happens that  $h\geq 0$, then the integrand is nonnegative and $\phi(t,p)\,z\gg 0$ for $t>0$. If not, just take $\gamma>0$ such that $h+\gamma\geq 0$ and note that, as it has been remarked in Caraballo et al.~\cite{calaobsa}, $\wit \phi(t,p)=e^{\gamma t}\phi(t,p)$ is the linear cocycle associated with the linear problems
\begin{equation*}
\left\{\begin{array}{l} \des\frac{\partial y}{\partial t}  =
 \Delta \, y+(h(p{\cdot}t,x)+\gamma)\,y\,,\quad t>0\,,\;\,x\in U, \;\, \text{for each}\; p\in P,\\[.2cm]
y=0\,,\quad  t>0\,,\;\,x\in \partial U.
\end{array}\right.
\end{equation*}
Since $h+\gamma\geq 0$, we are in the previous situation and thus $\wit \phi(t,p)\,z\gg 0$ for $t>0$, which implies that also $\phi(t,p)\,z\gg 0$ for $t>0$. The proof is finished.
\end{proof}
As a first consequence of this result, as $X^\alpha\hookrightarrow X$,
the linear skew-product semiflow $\tau_L:\R_+\times P\times X^\alpha\to P\times X^\alpha$ admits a continuous separation in the terms established before, replacing  $(X,\|\,{\cdot}\,\|)$ by $(X^\alpha,\|\,{\cdot}\,\|_\alpha)$.
\begin{nota}\label{nota-dirichlet}
For convenience in the notation, we denote by $\wit e(p)\gg 0$ with $\|\wit e(p)\|_\alpha=1$ the leading vectors in the principal bundle, and we take $e(p)=\wit e(p)/\|\wit e(p)\|\gg 0$ which satisfy $\|e(p)\|=1$. Then, we can associate to $h$ a 1-dim linear continuous cocycle $c(t,p)$ just as in~\eqref{c}. It is immediate to check that  the cocycle $c(t,p)$ is cohomologous to the cocycle $\wit c(t,p)$ determined by $\phi(t,p)\,\wit e(p)=\wit c(t,p)\,\wit e(p{\cdot}t)$ for $t\geq 0$ and $p\in P$, and extended for $t<0$ in the usual way. Also, for convenience we will write property (5) in the continuous separation as $\|\phi(t,p)\,z\|_\alpha\leq M \,e^{-\delta t}\|\phi(t,p)\,e(p)\|_\alpha$ for $t\geq 0$, $p\in P$ and $z\in X_2(p)$ with $\|z\|_\alpha=1$.
\end{nota}
At this point it is convenient to point out that all the results in Section~4 in~\cite{calaobsa} stated for scalar linear parabolic PDEs with null upper Lyapunov exponent with either Neumann or Robin boundary conditions also hold with Dirichlet boundary conditions with obvious minor modifications, taking $X^\alpha$ in the fiber when necessary. Some of them will be applied in the forthcoming sections with no further mention.
\par
Also, it is important to check that the value of the upper Lyapunov exponent is independent of the space $P\times C_0(\bar U)$ or $P\times X^\alpha$ considered for the linear semiflow. To see it, argue regardless of the uniqueness ergodic assumption, and note that by
Proposition~3 in Obaya and Sanz~\cite{obsa2019} the Lyapunov exponents $\lambda_s^\alpha(p,z)$ for $p\in P$ and $z\in X^\alpha$, $z\not= 0$ can also be calculated using the sup-norm on $X$, i.e.,
\[
\lambda_s^\alpha(p,z):= \limsup_{t\to \infty} \frac{\ln \|\phi(t,p)\,z\|_\alpha}{t}=\limsup_{t\to \infty} \frac{\ln \|\phi(t,p)\,z\|}{t}=:\lambda_s(p,z)\,,
\]
and that if $z\in X$, $z\not= 0$, also $\lambda_s(p,z)=\lambda_s^\alpha(p,z)$, since $\phi(t,p)\,z\in X^\alpha$ for $t>0$.
\par
As another consequence of Proposition~\ref{prop-strongly positive}, we get the strong monotonicity of the skew-product semiflow $\tau$, by linearizing.  Note that it is here where $g(p,x,y)$ must be of class $C^1$ in $y$, as it has been required. We omit the proof, since it is identical to the proof of Theorem~3.3 in~\cite{calaobsa} in the case of Neumann or Robin boundary conditions, except for now we apply the study of the variational equations in the case of Dirichlet boundary conditions, which has been developed in Section~4 in Obaya and Sanz~\cite{obsa2019} in a more general context with delay.
\begin{prop}
The skew-produc semiflow $\tau:\R_+\times P\times X\to P\times X^\alpha$ is strongly monotone, i.e., if $p\in P$ and $z_1, z_2\in X$ with $z_1<z_2$, then $u(t,p,z_1)\ll u(t,p,z_2)$ for $t>0$.
\end{prop}
\section{Forwards attraction properties for linear-dissipative problems with $\lambda_P=0$}\label{sec-forwards attraction}\noindent
In this section we consider a family of linear-dissipative problems~\eqref{pdefamilynl} over a minimal, uniquely ergodic and aperiodic flow $(P,\theta,\R)$ defined on a compact metric space $P$, with Neumann, Robin or Dirichlet boundary conditions, where $g$ satisfies (c1)-(c5) and $\lambda_P=0$ is assumed. Let $\A$ be the global attractor for  $\tau$.
\par
As it is standard, associated to the skew-product semiflow $\tau$, for each fixed $p\in P$ the related evolution process on $X$ is defined by $S_p(t,s)\,z=u(t-s,p{\cdot}s,z)$ for any $z\in X$ and $t\geq s$. Then, for each fixed $p\in P$, the family of compact sets $\{A(p{\cdot}t)\}_{t\in  \R}$ is the pullback attractor for the process  $S_p(\cdot,\cdot)$,
 meaning that:
\begin{itemize}
\item[(i)] it is invariant, i.e.,  $S_p(t,s)\,A(p{\cdot}s)=A(p{\cdot}t)$ for any $t\geq s\,$;
\item[(ii)] it pullback attracts bounded subsets of $X$, i.e., for any bounded set $B\subset X$,
\[
\lim_{s\to -\infty}{\rm dist}(S_p(t,s)\,B,A(p{\cdot}t))=0 \quad \hbox{for any}\; t\in \R\,;
\]
\item[(iii)] it is the minimal family of closed sets with property~(ii).
\end{itemize}
A nice reference for processes and pullback attractors is Carvalho et al.~\cite{calaro}.
\par
The main issue in this section is to study whether the processes have some forwards attraction properties too. More precisely, we would like to know for which $p\in P$, $\{A(p{\cdot}t)\}_{t\in  \R}$ is a forwards attractor for the process  $S_p(\cdot,\cdot)$, meaning that
\begin{equation*}
\lim_{t\to \infty}{\rm dist}(u(t,p,B),A(p{\cdot}t))=0   \quad \hbox{for any bounded set}\; B\subset X.
\end{equation*}
Note that the family of compact sets $\{A(p{\cdot}t)\}_{t\in  \R}$ might not be the minimal one with the previous forwards attracting property. To this respect, see Proposition~\ref{prop-lim sup 0}.
\begin{nota}\label{nota-forward atr}
In the case of Dirichlet boundary conditions, the concept of forwards attractor is independent of the space considered,  $X$ or $X^\alpha$, since the embedding $P\times X^\alpha\hookrightarrow P\times  X$ and the map $\tau_1:P\times X\to P\times X^\alpha$ are uniformly continuous over $\A$.
\end{nota}
For this purpose, we need to recall what we know about the global attractor $\A$. A precise description of  $\A$ for the case of null upper Lyapunov exponent is given in Caraballo et al.~\cite{calaobsa} for Neumann or Robin boundary conditions. The proofs can be easily adapted to the case of Dirichlet boundary conditions, using Proposition~\ref{prop-b Dirichlet} for the semicontinuity of $b:P\to X^\alpha$, and Theorem~7 in Cardoso  et al.~\cite{cardoso}, whose proof also works in the Dirichlet case.  Namely, the maps
\[
a(p)=\inf A(p)\quad\text{ and } \quad b(p)=\sup A(p) \quad\text{for any}\;\,p\in P,
\]
define semicontinuous equilibria for $\tau$. Condition (c4) is assumed for the sake of simplicity, since then $a(p)= -b(p)$, and
$\A\subseteq \bigcup_{p\in P} \{p\}\times [-b(p),b(p)]$.
The inner structure of $\A$ is firstly classified in two types, roughly speaking, either wide or pinched, depending on to which of the sets
\begin{align*}
B(P\times \bar U) &= \{ h\in C_0(P\times\bar U)\mid \sup_{t\in\R} | \ln c(t,p)|<\infty \;\text{ for any}\;p\in P \}\,\; \text{or} \\
\mathcal{U}(P\times \bar U) &=C_0(P\times\bar U)\setminus B(P\times \bar U) \,
\end{align*}
the map $h$ in $C_0(P\times\bar U)=\{ h\in C(P\times\bar U)\mid \lambda_P(h)=0\}$, which is the linear coefficient of the problems, belongs. Here $c(t,p)$ is the 1-dim linear cocycle  given in~\eqref{c} (see Remark~\ref{nota-dirichlet} in the Dirichlet case and note that cohomologous cocycles have the same behaviour in what refers to boundedness, in the sense that $\sup_{t\in\R} |\ln  c(t,p)|<\infty$ if and only if $\sup_{t\in\R} |\ln \wit c(t,p)|<\infty$).
For the sake of completeness, we include the statements  of  Theorem~5.1 and Theorem~5.2 in~\cite{calaobsa} with the precise description of $\A$ in both cases, but first we make a remark on the notation used hereafter.
\begin{nota}
As before, $X$ stands for $C(\bar U)$ in the Neumann and Robin cases, whereas it stands for $C_0(\bar U)$ in the Dirichlet case, always with the sup-norm $\|\,{\cdot}\,\|$. To unify the writing, sometimes we will write $(X^\gamma,\|\,{\cdot}\,\|_\gamma)$, meaning $(X,\|\,{\cdot}\,\|)$ in the Neumann and Robin cases, but $(X^\alpha,\|\,{\cdot}\,\|_\alpha)$ in the Dirichlet case, as defined in Section~\ref{subsect-Dirichlet}. The order relations in $X$ and $X^\gamma$ will just be denoted by $\leq$, $<$ and $\ll$ according to~\eqref{order}, but have in mind the different spaces involved in each case.
\end{nota}
\begin{teor}\label{teor-estr atractor caso b}
$($\cite{calaobsa}$)$ Let $h \in B(P\times\bar U)$ and let $\widehat e: P \to \Int X^\gamma_+$ be a  continuous equilibrium map for the linear semiflow $\tau_L$.
Then, there exists an $r_*>0$ such that
\[
A(p)=\{r\,\widehat e(p)\mid |r|\leq r_*\}\subset X_1(p)\quad\text{for any}\;\,p\in P,
\]
for $X_1(p)$ the 1-dim subspace of $X^\gamma$ given by the section of the principle bundle. Besides, in this case the upper boundary map $b$ is continuous.
\end{teor}
The continuity of the section map $p\mapsto A(p)$ in this situation makes it possible to apply Theorem~4.3 in Cheban et al.~\cite{chks} to get the forwards attraction for all $p\in P$.
\begin{teor}\label{teor-atraccion forward-caso B}
Let $h\in B(P\times\bar U)$. Then,  $\lim_{t\to\infty} {\rm dist}(u(t,p,B),A(p{\cdot}t))=0$
for any bounded set $B\subset X$, uniformly for $p\in P$.
\end{teor}
With the former result,  the study for $h\in B(P\times\bar U)$ is closed.
\begin{teor}\label{teor-estr atractor caso u}
$($\cite{calaobsa}$)$ Let $h \in \mathcal{U}(P\times\bar U)$. Then, the global attractor $\A$ is a pinched set. More precisely:
\begin{itemize}
\item[(i)] There exists an invariant residual set $P_{\rm{s}}\subsetneq P$ such that $b(p)=0$ for any $p\in P_{\rm{s}}$. In fact $P_{\rm{s}}$ is the set of continuity points of $b$.
\item[(ii)] The set $P_{\rm{f}}=P\setminus P_{\rm{s}}$ is an invariant dense set of first category and $b(p)\gg 0$ for any $p\in P_{\rm{f}}$.
\end{itemize}
\end{teor}
\begin{nota}\label{nota-1}
(i) The sets $P_{\rm{s}}$ and $P_{\rm{f}}$  exclusively depend on the linearized problems~\eqref{linealizada}, and more precisely on the behaviour for $t\leq 0$  of the 1-dim linear cocycle $c(t,p)$ defined in~\eqref{c} (see Proposition~5.3 in~\cite{calaobsa}): $p\in P_{\rm{f}} \Leftrightarrow \sup_{t\leq 0} c(t,p)<\infty$. This characterization makes the sets $P_{\rm{s}}$ and $P_{\rm{f}}$ invariant under cohomology of the cocycle $c(t,p)$. Recall also that, since the principal spectrum $\Sigma_{\text{pr}}=\{\lambda_P\}=\{0\}$, there is no exponential dichotomy in the principal bundle and then there exists a $p_0\in P$ such that $\sup_{t\in\R} c(t,p_0)<\infty$, so that $P_{\rm{f}}\not=\emptyset$ (see \cite{calaobsa} for more details).
\par
(ii) For $p\in P_{\rm{f}}$, there is a nontrivial segment in the attractor inside the principal bundle, namely, $\{(p,re(p))\mid 0<r\leq r_0/m(p)\}\subset \A\cap (P\times \Int X^\gamma_+)$ for the positive constant $m(p)=\sup_{t\leq 0} c(t,p)$. Just note that for $0<r\leq r_0/m(p)$, the map $z(s)=rc(s,p)\,e(p{\cdot}s)$ for $s\leq 0$ provides a backward semiorbit for $\tau_L$, i.e., $\phi(t,p{\cdot}s)\,z(s)=z(t+s)$ for $s\leq 0$ and  $0\leq t\leq -s$. Besides, it remains in the linear zone of the problem, thus, it is a solution of the nonlinear problem too, and we can continue this solution to have a bounded entire orbit, so that $(p,re(p))\in\A$.
\end{nota}
In the next results, we follow the spirit and ideas of Section~4 in Caraballo et al.~\cite{caloNonl}, investigating the possible forwards attraction of the (pullback) cocycle attractor.
In correspondence with the definitions of recurrent and asymptotic points for  maps in  $C_0(P)=\left\{a\in C(P)\mid \int_P a\,d\nu=0\right\}$ (see~\cite{caloNonl}), where $\nu$ is the only invariant measure for the flow in $P$, we include the definitions of recurrent and asymptotic points for maps $h\in C_0(P\times \bar U)$, in terms of its associated 1-dim linear cocycles.
\begin{defi}
Let $h\in C_0(P\times \bar U)$ and let $c(t,p)$ be the 1-dim linear cocycle defined in~\eqref{c}. \par
(i) $p\in P$ is said to be (Poincar\'{e}) {\em recurrent\/} at $\infty$ (resp.~at $-\infty$) for $h$ if there exists a sequence $(t_n)_n\uparrow \infty$ (resp.~$(t_n)_n\downarrow -\infty$) such that $\des\lim_{n\to\infty} c(t_n,p)=1$.
\par
(ii) We denote by $P_{\rm{a}}^+$ (resp.~$P_{\rm{a}}^-$) the set of {\em asymptotic  points\/} at $\infty$ (resp.~at $-\infty$) for $h$, that is, the points $p\in P$  with $\des\lim_{t\to\infty} c(t,p)=0$ (resp.~$\des\lim_{t\to -\infty} c(t,p)=0$).
\par
(iii) We denote by $P_{\rm{o}}$  the invariant and residual set of {\em oscillating points\/} given in Theorem~4.5 in~\cite{calaobsa}. For any $p\in P_{\rm{o}}$,  $\des\liminf_{t\to\pm\infty} c(t,p)=0$ and $\des\limsup_{t\to\pm\infty} c(t,p)=\infty$.
\end{defi}
\begin{nota}\label{nota-recurrentes}
(i) If $h\in \mathcal{U}(P\times\bar U)$, then there is an invariant set $P_{\rm{r}}^+$ of full measure whose elements are recurrent points at $\infty$:
\begin{equation}\label{pr}
P_{\rm r}^+=\{p\in P \mid p{\cdot}t \;\text{is recurrent at $\infty$ for every}\;t\in\R \}\,,
\end{equation}
as well as
an invariant set $P_{\rm{r}}^-$ of full measure whose elements are recurrent points at $-\infty$, defined in the same fashion.
Besides, since the set $P_{\rm{o}}$ of oscillating points  is contained in  $P_{\rm{r}}=P_{\rm{r}}^+\cap P_{\rm{r}}^-$, the sets $P_{\rm{r}}^+$, $P_{\rm{r}}^-$ and $P_{\rm{r}}$ are also residual. The reader is referred to the proof of Proposition~5.5 in~\cite{calaobsa} for all the details.
\par
(ii) The sets $P_{\rm{a}}^+$ and $P_{\rm{a}}^-$ are invariant: use the cocycle relation. Also $P_{\rm{a}}^-\subset P_{\rm{f}}$.
\par
(iii) The sets $P_{\rm{a}}^+$, $P_{\rm{a}}^-$ and $P_{\rm{o}}$ are invariant under cocycle cohomology.
\end{nota}
In the next result we prove that the asymptotic points at $\infty$ for $h \in \mathcal{U}(P\times\bar U)$ are exactly those for which positive mild solutions starting in $\Int X^\gamma_+$ of both the linear and the nonlinear abstract problems go to $0$ as $t\to\infty$.
\begin{prop}\label{prop-asymptotic points}
Let $h \in \mathcal{U}(P\times\bar U)$. Then, for $p\in P$ the following conditions are equivalent:
\begin{itemize}
  \item[(i)] $\des\lim_{t\to\infty}\phi(t,p)\,z=0$ for some $z\gg 0$ (and thus for any $z\gg 0$);
  \item[(ii)] $\des\lim_{t\to\infty}u(t,p,z)=0$ for some $z\gg 0$ (and thus for any $z\gg 0$);
  \item[(iii)] $p\in P_{\rm{a}}^+$.
\end{itemize}
\end{prop}
\begin{proof}
(i)$\Rightarrow$(ii) is clear, since $0\ll u(t,p,z)\leq \phi(t,p)\,z$ for every $t\geq 0$ (see Theorem~3.1 in~\cite{calaobsa} for the result of comparison of solutions).
\par
(ii)$\Rightarrow$(iii): fixed a $z\gg 0$ with $\lim_{t\to\infty}u(t,p,z)=0$, there exists a $t_0$ such that $u(t,p,z)\leq \bar r_0$ for any $t\geq t_0$, where $\bar r_0$ is the map on $\bar U$ identically equal to $r_0$, the constant in (c5) delimiting the linear zone of the problems. Then, for $t\geq 0$,
\[
u(t+t_0,p,z)=u(t,p{\cdot}t_0,u(t_0,p,z))= \phi(t,p{\cdot}t_0)\,u(t_0,p,z)\geq \lambda\,  \phi(t,p{\cdot}t_0)\,e(p{\cdot}t_0)
\]
provided that $\lambda>0$ is small enough so that  $u(t_0,p,z)\geq \lambda\, e(p{\cdot}t_0)$. Then, as $t\to\infty$, $\phi(t,p{\cdot}t_0)\,e(p{\cdot}t_0)=c(t,p{\cdot}t_0)\,e(p{\cdot}(t_0+t))\to 0$, that is,  $p{\cdot}t_0\in P_{\rm{a}}^+$ and thus $p\in P_{\rm{a}}^+$.
\par
(iii)$\Rightarrow$(i): fixed any $z\gg 0$, take an $r>0$ large enough so that $z\leq re(p)$. Then, $\phi(t,p)\,z\leq r\phi(t,p)\,e(p)=r c(t,p)\,e(p{\cdot}t)\to 0$ as $t\to\infty$. The proof is finished.
\end{proof}
We state a first partial result on forwards attraction. When we write the trivial forwards attractor $\{0\}$, we mean the family of compact sets of $X$ given by $\{0\}_{t\in\R}$.
\par
\begin{prop}\label{prop-forward attractor}
Let $h \in \mathcal{U}(P\times\bar U)$, and let $P_{\rm{r}}^+$ be the set in~\eqref{pr}. Then:
\begin{itemize}
\item[(i)] $p\in P_{\rm{a}}^+$ if and only if the process $S_p(\cdot,\cdot)$ has the forwards attractor $\{0\}$.
\item[(ii)] If $p\in P_{\rm{s}}\cap P_{\rm{r}}^+$, then the process $S_p(\cdot,\cdot)$ has no forwards attractor.
\item[(iii)] If $p\in P_{\rm{o}}$, then the process $S_p(\cdot,\cdot)$ has no forwards attractor.
\end{itemize}
\end{prop}
\begin{proof}
The proof of~(i) is a corollary of Proposition~\ref{prop-asymptotic points}. Just note that given a bounded set $B\subset X$, $\tau_1(P\times B)$ is relatively compact in $P\times X^\gamma$ and we can take a $z_0\gg 0$ such that $-z_0\leq u(1,p,z)\leq z_0$ for any $p\in P$ and $z\in B$. Then apply the monotonicity and the {\em odd\/} character of the semiflow. As for~(ii), Theorem~\ref{teor-estr atractor caso u} says that if $p\in P_{\rm{s}}$, the only chance for a forwards attractor is $\{0\}$. However, since $P_{\rm{a}}^+\cap P_{\rm{r}}^+=\emptyset$,~(i) precludes the existence of a forwards attractor for the process given for $p\in P_{\rm{s}}\cap P_{\rm{r}}^+$. Finally~(iii) follows from~(ii), since $P_{\rm{o}}\subset P_{\rm{s}}\cap P_{\rm{r}}^+$.
\end{proof}
Note that we detect the lack of forwards attraction in continuity points of $b$. It can happen that there is no forwards attraction for all continuity points of $b$. This occurs for instance when every point $p\in P$ is recurrent at $\infty$, so that $P_{\rm{r}}^+=P$. We will return to this matter later: see Example~\ref{eje-kozlov}.
\par
The next result follows from Remark~\ref{nota-recurrentes}~(i)  and Proposition~\ref{prop-forward attractor}~(ii).
\begin{coro}
Let $h \in \mathcal{U}(P\times\bar U)$. If $\nu(P_{\rm{s}})=1$, then there exists a residual invariant set of full measure $P_{\rm{s}}^*=P_{\rm{s}}\cap P_{\rm{r}}^+$ such that, if  $p\in P_{\rm{s}}^*$, the process $S_p(\cdot,\cdot)$ has no forwards attractor.
\end{coro}
Now the natural question is what can be said when $\nu(P_{\rm{f}})=1$. We will prove that the pullback attractor $\{A(p{\cdot}t)\}_{t\in  \R}$ is also a forwards attractor for almost all the processes $S_p(\cdot,\cdot)$. We begin by analysing where the attractor is located with respect to the linear zone of the problems, in terms of the behaviour of the trajectories along the upper boundary map $b$.  Recall that if $p\in P_{\rm{s}}$, then $b(p{\cdot}t)=0$ for any $t\in \R$, so that the focus is on the behaviour when $p\in P_{\rm{f}}$.
First of all, the  asymptotic points  at $\infty$ in $P_{\rm{f}}$ are characterized by $\lim_{t\to \infty} \|b(p{\cdot}t)\|=0$.
\begin{prop}\label{prop-lim sup 0}
Let $h\in\mathcal{U}(P\times\bar U)$. For $p\in P_{\rm{f}}$, $\lim_{t\to \infty} \|b(p{\cdot}t)\|=0$ if and only if $p\in P_{\rm{a}}^+$; then, $\{0\}$ is the (minimal) forwards attractor for the process $S_{p}(\cdot,\cdot)$.
\end{prop}
\begin{proof}
Just note that $b(p)\gg 0$, and $u(t,p,b(p))=b(p{\cdot}t)$, $t\geq 0$, so that the result follows from Proposition~\ref{prop-asymptotic points} and Proposition~\ref{prop-forward attractor}~(i).
\end{proof}
Some general properties for $p_0\in P_{\rm{f}}$ are the following.
\begin{prop}\label{prop-norma b}
Let $h\in\mathcal{U}(P\times\bar U)$. Then, for every $p_0\in P_{\rm{f}}$,
\begin{itemize}
\item[(i)]  $\des\liminf_{t\to -\infty} \|b(p_0{\cdot}t)\|=0$ and $\,\des\limsup_{t\to -\infty} \|b(p_0{\cdot}t)\|\geq r_0$;
\item[(ii)] $\des\liminf_{t\to \infty} \|b(p_0{\cdot}t)\|=0$.
\end{itemize}
\end{prop}
\begin{proof}
Let $\delta >0$ and assume that there exists a $t_0>0$ such that $\|b(p_0{\cdot}t)\|\geq \delta$ for $t\leq -t_0$. Then, consider $K$ the $\alpha$-limit set of $(p_0,b(p_0))$ which satisfies that for any $p\in P$ there is a $(p,z)\in K$ and $\|z\|\geq \delta$. Since there are (unique, see Teman~\cite{tema}) backward extensions inside $K$, we can build a bounded entire orbit through $(p,z)$ which is necessarily included in the attractor. Then, $z\in A(p)$ with $\|z\|\geq \delta$. Since chances are $b(p)=0$ or $b(p)\gg 0$, it must be $b(p)\gg 0$ for any $p\in P$, but this cannot happen according to Theorem~\ref{teor-estr atractor caso u}.
As a consequence, taking $\delta>0$ as small as wanted, we get that $\liminf_{t\to -\infty} \|b(p_0{\cdot}t)\|=0$ .
\par
Suppose now that there exist a $\rho$, with $0<\rho<r_0$, and a $t_0>0$ such that $\|b(p_0{\cdot}t)\|\leq \rho<r_0$ for $t\leq -t_0$. Then, take $\lambda= r_0/\rho$ and look at $\lambda\, b(p_0{\cdot}t)$. Note that since for $t\leq -t_0$, $b(p_0{\cdot}t)$ remains in the linear zone of the problem, also $\lambda\, b(p_0{\cdot}t)$ is a solution of the linear problem for $t\leq -t_0$, and since it satisfies $\|\lambda\, b(p_0{\cdot}t)\|\leq r_0$ for $t\leq -t_0$, it is also a bounded solution of the nonlinear problem for $t\leq -t_0$. If we continue this solution of the nonlinear problem forwards, we have a bounded entire orbit which necessarily lies inside the attractor. But this is a contradiction, since $b$ is the upper boundary map of the attractor and $\lambda>1$. Thus, taking $\rho$ as close to $r_0$ as wanted, we get that
$\limsup_{t\to -\infty} \|b(p_0{\cdot}t)\|\geq r_0$.
\par
Finally, assume that there are a $\delta>0$ and a $t_0>0$   such that $\|b(p_0{\cdot}t)\|\geq \delta$ for $t\geq t_0$. This time we consider $K$ the $\omega$-limit set of $(p_0,b(p_0))$ which satisfies that for any $p\in P$ there is a $(p,z)\in K$ and $\|z\|\geq \delta$. Arguing exactly as in the first paragraph of the proof, we get a contradiction.  As a consequence, taking $\delta>0$ as small as wanted, we get that $\liminf_{t\to \infty} \|b(p_0{\cdot}t)\|=0$. The proof is finished.
\end{proof}
Note that this result prevents the possibility that there might be some $p\in P$ for which $\|b(p{\cdot}t))\|>r_0$ for any $t\geq t_0$ or any $t\leq -t_0$, for some $t_0>0$. However, the map $\R\to \R,\; t \mapsto \|b(p_0{\cdot}t)\|$ might have  a recurrent crossing  behaviour  with respect to the threshold $r_0$. We determine some conditions on  the 1-dim cocycle $c(t,p_0)$ which guarantee this fact. Note that if a crossing  behaviour is to be expected, it must be $p_0\in  P_{\rm{f}}$. In any case, we insist on the fact that this can only happen rarely, since if $\nu(P_{\rm{s}})=1$, then for almost every $p_0\in P$, $\|b(p_0{\cdot}t)\|=0$ for any $t\in\R$; whereas  if $\nu(P_{\rm{f}})=1$, then $\|b(p_0{\cdot}t)\|$ never overpasses the threshold $r_0$ for $p_0\in P_{\rm{f}}\cap P_{\rm{r}}$, with $\nu(P_{\rm{r}})=1$: see Remark~\ref{nota-recurrentes}~(i) and Theorem~\ref{teor-limsup}.
\begin{prop}\label{prop-crossing}
Let $h\in\mathcal{U}(P\times\bar U)$ and let $c(t,p)$ be the associated real cocycle.
\begin{itemize}
\item[(i)] If for some $p_0\in P$, $\lim_{t\to -\infty} c(t,p_0)=0$, then  there exists a sequence $(t_n^1)_n\downarrow -\infty$ such that $\|b(p_0{\cdot}t_n^1)\|>r_0$ for any $n\geq 1$.
\item[(ii)] If for some $p_0\in P$, $\limsup_{t\to \infty} c(t,p_0)=\infty$, then there exists a sequence $(t_n^2)_n\uparrow \infty$ such that $\|b(p_0{\cdot}t_n^2)\|>r_0$ for any $n\geq 1$.
\end{itemize}
\end{prop}
\begin{proof}
(i) Argue by contradiction and assume that for some $t_0>0$ it holds that
$\|b(p_0{\cdot}t)\|\leq r_0$ for $t\leq -t_0$. Then, $b(p_0{\cdot}t)$ is a bounded  solution of the abstract linear problem for $t\leq -t_0$, and according to Proposition~4.11~(i) in~\cite{calaobsa}, $b(p_0{\cdot}(-t_0))\in X_1(p_0{\cdot}(-t_0))$, that is, $b(p_0{\cdot}(-t_0))=\beta\, e(p_0{\cdot}(-t_0))$ for some $\beta>0$. Since $c(t,p)$ determines the linear dynamics in the principal bundle, for $s\leq 0$ we can write
$
b(p_0{\cdot}(-t_0+s))= \beta\,c(s,p_0{\cdot}(-t_0))\,e(p_0{\cdot}(-t_0+s))=\frac{\beta\,c(-t_0+s,p_0)}{c(-t_0,p_0)}\,e(p_0{\cdot}(-t_0+s)).
$
From here it follows that $\lim_{t\to -\infty} \|b(p_0{\cdot}t)\|=0$, but this is a contradiction of Proposition~\ref{prop-norma b}~(i), and we are done.
\par
(ii) If $\limsup_{t\to \infty} c(t,p_0)=\infty$, once more argue by contradiction and assume that for some $t_0>0$,
$\|b(p_0{\cdot}t)\|\leq r_0$ for $t\geq t_0$. This time $b(p_0{\cdot}t)$ is a solution of the abstract linear problem for $t\geq t_0$. Take $\beta>0$ so that $e(p_0{\cdot}t_0)\leq \beta\,b(p_0{\cdot}t_0)$. By monotonicity, $\phi(t,p_0{\cdot}t_0)\,e(p_0{\cdot}t_0)\leq \beta\,\phi(t,p_0{\cdot}t_0)\,b(p_0{\cdot}t_0)=\beta\,b(p_0{\cdot}(t+t_0))$ for $t\geq 0$. Now, $
\phi(t,p_0{\cdot}t_0)\,e(p_0{\cdot}t_0)=c(t,p_0{\cdot}t_0)\,e(p_0{\cdot}(t+t_0))=\frac{c(t+t_0,p_0)}{c(t_0,p_0)}\,e(p_0{\cdot}(t+t_0))$,
and therefore, $c(t+t_0,p_0)\leq c(t_0,p_0)\,\beta\,\|b(p_0{\cdot}(t+t_0))\|$ for any $t\geq 0$, which is in contradiction with the hypothesis. The proof is finished.
\end{proof}
\begin{eje}
It is important to underline that, based on precise examples of almost periodic functions given in the literature, we get the evidence that  all the situations covered in the paper occur. Namely, given an almost periodic map $a_0:\R\to\R$  with zero mean value, $0=\lim_{t\to \infty}(1/t)\int_0^t a_0(s)\,ds$, and unbounded integral $\int_0^t   a_0(s)\,ds$ ($t\in\R$), one can build the hull of $a_0$, $P\subset C(\R)$, as it has been explained in Section~\ref{sec-preli}; define the shift flow on $P$ just denoted by $p{\cdot}t$; and consider the continuous map $a\in C_0(P)=\left\{f\in C(P)\mid \int_P f\,d\nu=0\right\}$ defined by $a:P\to \R$, $p\mapsto p(0)$, in such a way that $a(p{\cdot}t)=p(t)$ for $t\in \R$. In particular, for $p=a_0$ we recover the initial almost periodic map $a_0$.
\par
Now, as noted in Remark~\ref{nota-cohomologos}, given the smooth cocycle  $c(t,p)=\exp\int_0^t a(p{\cdot}s)\,ds$ and $h\in C_0(P\times \bar U)$, we can find a $k\in C_0(P)$ such that the 1-dim cocycle associated to the family of problems~\eqref{pdefamily h+k} is cohomologous to $c(t,p)$. Here recall that the bounded or unbounded character of a cocycle, and thus the sets $P_{\rm{f}}$, $P_{\rm{s}}$, $P_{\rm{a}}^+$, $P_{\rm{a}}^-$ and $P_{\rm{o}}$ are invariant under cocycle cohomology. Therefore, it suffices to have precise examples for $p=a_0$  of different behaviours of $\exp\int_0^t a_0(s)\,ds$ ($t\in\R$).
\par
For instance, one can build an example of an oscillating point $a_0\in P_{\rm{o}}$ out of an example given in Poincar\'{e}~\cite{poin}. Note that oscillating points are in particular recurrent. More examples can be found in Poincar\'{e}~\cite{poin86}. Also, Example~3.2.1 in Johnson~\cite{john80} offers a map $a\in C_0(\T^2)$ with unbounded integral such that for a.e.~$p\in \T^2$, $\sup_{t\in \R}\int_0^t a(p{\cdot}s)\,ds<\infty$ and thus, for the associated smooth cocycle, $\nu(P_{\rm{f}}\cap P_{\rm{r}})=1$ (see also Ortega and Tarallo~\cite{orta}).
\par
Examples of almost periodic  functions $a_0(t)$ with zero mean value and whose integral $\int_0^t  a_0$ grows like $t^\beta$ as $t\to \infty$ for some $0<\beta<1$  have been explicitly built in the literature by several authors, such as Poincar\'{e}~\cite{poin}, Zhikov and Levitan~\cite{zhle} and Johnson and Moser~\cite{jomo}. In this case, $\lim_{t\to\infty} \exp \int_0^t a_0=\infty$. This behaviour is compatible with the even or odd character of the map $a_0$. If the map $a_0(t)$ is  even, then $\lim_{t\to-\infty} \exp \int_0^t a_0=0$, so that $p=a_0\in  P_{\rm{a}}^-\subset P_{\rm{f}}$ and $\|b(p{\cdot}t)\|$ crosses the threshold $r_0$ infinitely many times as $t\to\pm\infty$; whereas if the map $a_0$ is odd, then $\lim_{t\to-\infty} \exp \int_0^t a_0=\infty$ and  $p=a_0\in P_{\rm{s}}$.   Also, this time considering the almost periodic map with zero mean value $\widetilde a_0(t)= a_0(-t)$, $t\in\R$ it holds that $\lim_{t\to -\infty} \exp \int_0^t   \wit a_0 =0$ and, for the corresponding hull, $p=\wit a_0\in P_{\rm{a}}^-\subset P_{\rm{f}}$ and $\|b(p{\cdot}t)\|$ crosses the threshold $r_0$ infinitely many times as $t\to -\infty$. In this case, if  $a_0$ is even, so is $\wit a_0$ and thus, $\lim_{t\to\infty} \exp \int_0^t   \wit a_0=\infty$ and $\|b(p{\cdot}t)\|$ crosses the threshold $r_0$ infinitely many times also as $t\to +\infty$; whereas if $a_0$ is odd, so is $\wit a_0$ and thus, $\lim_{t\to\infty} \exp \int_0^t   \wit a_0=0$, that is, $p=\wit a_0\in P_{\rm{a}}^+$ and $\lim_{t\to\infty}\|b(p{\cdot}t)\|=0$.
\end{eje}
The 1-dim cocycle $c(t,p_0)$ has some crucial particular properties for elements $p_0$ in $P_{\rm{f}}$ with a recurrent orbit at $\pm \infty$ (see Remark~\ref{nota-recurrentes}~(i)).
\begin{lema}\label{lema-recurrentes}
If $p_0\in P_{\rm{f}}\cap P_{\rm{r}}$,
$\des\limsup_{t\to \infty}  c(t,p_0)=\des\limsup_{t\to -\infty}  c(t,p_0)=\des\sup_{t\in\R} c(t,p_0)<\infty$.
\end{lema}
\begin{proof}
If for a sequence $(t_n^1)_n\downarrow -\infty$,  $\lim_{n\to\infty}c(t_n^1,p_0)=r \in [0,\infty]$, since for any $n\geq 1$, $p_0{\cdot}t_n^1$ is recurrent at $\infty$, we can find a sequence $(t_n^2)_n\uparrow \infty$ such that $c(t_n^2-t_n^1,p_0{\cdot}t_n^1)\to 1$ as $n\to\infty$. Then, by the cocycle property, $c(t_n^2,p_0)=c(t_n^2-t_n^1,p_0{\cdot}t_n^1)\,c(t_n^1,p_0)\to r$ as $n\to \infty$, and so $\limsup_{t\to -\infty}  c(t,p_0)\leq \limsup_{t\to \infty}  c(t,p_0)$. The same argument, using this time the recurrence at $-\infty$,  shows the converse inequality, so that $\limsup_{t\to -\infty}  c(t,p_0)=\limsup_{t\to \infty}  c(t,p_0)\leq \sup_{t\in \R} c(t,p_0)$. We know that $\sup_{t\leq 0} c(t,p_0)<\infty$ because $p_0\in P_{\rm{f}}$, so that the superior limits at $\pm\infty$ are finite and, by continuity, also  $\sup_{t\in \R} c(t,p_0)<\infty$. If the superior $\sup_{t\in \R} c(t,p_0)$ were attained at some $t_0\in\R$ , then by the recurrence of $p_0{\cdot}t_0$ at $\infty$, we could take a sequence $(t_n)_n\uparrow \infty$ with $\lim_{n\to\infty}c(t_n-t_0,p_0{\cdot}t_0)=1$, so that $\lim_{n\to\infty}c(t_n,p_0)=c(t_0,p_0)$. From this, it is easy to conclude the proof.
\end{proof}
For $p_0\in P_{\rm{f}}\cap P_{\rm{r}}$, $\pm b(p_0{\cdot}t)$ remain in the linear zone of the problems and have precise attracting properties. We focus on $b$ and the zone above it, but the corresponding results are true for $-b$ and the zone below it.
\begin{teor}\label{teor-limsup}
Let $h\in\mathcal{U}(P\times\bar U)$ and let $p_0\in P_{\rm{f}}\cap P_{\rm{r}}$. Then:
\begin{itemize}
\item[(i)] $0\ll b(p_0{\cdot}t)\leq \bar r_0$ for $t\in\R$, for $\bar r_0\equiv r_0$ on $\bar U$, and $\des\limsup_{t\to \infty} \|b(p_0{\cdot}t)\|=r_0$.
\item[(ii)] Given $z_0\gg 0$ such that $b(p_0)\leq z_0$, $\des\lim_{t\to\infty} u(t,p_0,z_0)-b(p_0{\cdot}t)=0$.
\item[(iii)] $A(p_0)\subset X_1(p_0)$, the 1-dim subspace of $X$ determined by the section of the principle bundle at the point  $p_0$. More precisely, $A(p_0)=\{\beta\,e(p_0)\mid |\beta|\leq \eta\}$ for an $\eta=\eta(p_0)>0$.
\end{itemize}
\end{teor}
\begin{proof}
(i) Let $k=\sup_{t\in \R} c(t,p_0)\in (0,\infty)$ and let us set  $u(t)=\frac{r_0}{k}\,c(t,p_0)\,e(p_0{\cdot}t)$ for $t\in \R$. It is easy to check that $(p_0{\cdot}t,u(t))$, $t\in\R$ defines an entire orbit for the linear semiflow $\tau_L$, that is, for any $t\geq 0$ and any $s\in \R$, $\phi(t,p_0{\cdot}s)\,u(s)=u(t+s)$. Besides,
$0\ll u(t)\leq \bar r_0$ for $t\in \R$, which means that $u(t)$ is also an entire bounded mild solution of the nonlinear problem~\eqref{pdefamilynl} for $p_0$, and therefore, $(p_0,\frac{r_0}{k}\,e(p_0))\in \A$ and $u(t)\leq b(p_0{\cdot}t)$ for $t\in\R$.  Also, by Lemma~\ref{lema-recurrentes}, $l_{\rm s}^+:=\limsup_{t\to\infty} \|u(t)\|=\limsup_{t\to\infty} \frac{r_0}{k}\,c(t,p_0) =r_0=\limsup_{t\to -\infty} \frac{r_0}{k}\,c(t,p_0)= \limsup_{t\to -\infty} \|u(t)\|=:l_{\rm s}^- $.
\par
If we prove that $b(p_0{\cdot}t)= u(t)$ for $t\in\R$, we are done with (i). For that, for each $t\in\R$ let us define the real number
\begin{equation*}
\lambda(t)=\inf\{\lambda\geq 1\mid b(p_0{\cdot}t)\leq \lambda\, u(t)\}\,.
\end{equation*}
Recalling that $u(t,p,z)\leq \phi(t,p)\,z$ for any $p\in P$, $z\geq 0$ and $t\geq 0$, and applying the monotonicity of $\tau$, we check that $\lambda(t)$ is a nonincreasing map; for if $b(p_0{\cdot}t)\leq \lambda\, u(t)$ at some $t\in\R$, then for $s>0$,  $b(p_0{\cdot}(t+s))=u(s, p_0{\cdot}t,b(p_0{\cdot}t))\leq   u(s, p_0{\cdot}t, \lambda\, u(t))\leq \phi(s, p_0{\cdot}t)\,\lambda\, u(t)=\lambda\, u(t+s)$, and thus  $\lambda(t+s)\leq \lambda(t)$.
\par
Let us now check that $\lim_{t\to -\infty}\lambda(t)<\infty$. Note that by monotonicity, it suffices to find a sequence $(t_n)_n\downarrow -\infty$ for which $\lambda(t_n)$ keeps bounded above. With this aim, we consider the set $\B=\{(p,z)\in\A\mid z\geq  0 \; \text{and}\; \|u(1,p,z)\|\geq r_0/2\}$, which is trivially closed, and since $\B\subset \A$, $\B$ is compact. Note that $u(1,p,0)=0$, so that $z>0$ for  $(p,z)\in \B$.  Thus, $\tau_1(\B)$ is a compact set in $P\times \Int X^\gamma_+$, by the strong monotonicity. Then, for an {\em a priori\/} fixed  $e_0\gg 0$, there exists a sufficiently small $0<\eta_1<1$ so that $\eta_1 e_0\leq z$ for any $(p,z)\in \tau_1(\B)$, and a sufficiently big $\eta_2>1$ so that $b(p)\leq \eta_2 e_0$ for  $p\in P$. At this point, since $l_{\rm s}^- =r_0$, there exists a sequence $(t_n)_n\downarrow -\infty$ such that $\|u(t_n)\|\geq r_0/2$. Then, by the construction, $(p_0{\cdot}t_n,u(t_n))\in \tau_1(\B)$ and then, $b(p_0{\cdot}t_n)\leq \eta_2 e_0\leq (\eta_2/\eta_1)\, u(t_n)$ for any $n\geq 1$, and thus, $\lambda(t_n)\leq \eta_2/\eta_1$ for any $n\geq 1$, as we wanted.
\par
Write $\lambda_0=\lim_{t\to -\infty}\lambda(t)\geq 1$. If $\lambda_0=1$, then $\lambda(t)\equiv 1$ and $b(p_0{\cdot}t)\leq u(t)$ for $t\in\R$. Since we had that $u(t)\leq b(p_0{\cdot}t)$ for $t\in\R$, then $b(p_0{\cdot}t)= u(t)$ for $t\in\R$. So, to finish the proof, argue by contradiction and assume that $\lambda_0>1$. Then, there exist $\delta_0, \delta_1>0$ small enough so that $\lambda_0(r_0-\delta_1)>r_0+\delta_0$. As $l_{\rm s}^- =r_0$, there exists a sequence $(t_n)_n\downarrow -\infty$ such that $\|u(t_n)\|\geq r_0-\delta_1$, $n\geq 1$. Now for  $\{(p_0{\cdot}t_n, u(t_n))\mid n\geq 1\}\subset \A$  we can assume without loss of generality that $(p_0{\cdot}t_n, u(t_n))\to (p_1,z_1)\in \A$ as $n\to \infty$. In particular $\|z_1\|\geq r_0-\delta_1$ and $\|\lambda_0z_1\|> r_0+\delta_0$. Thus, comparing the solutions of the nonlinear and the linear problems starting at $(p_1,\lambda_0z_1)$, there exists an $\varepsilon_1>0$ (actually for any $\varepsilon_1>0$) such that $u(\varepsilon_1,p_1,\lambda_0 z_1)\ll \phi(\varepsilon_1,p_1)\,\lambda_0 z_1$. By continuity, we can take $1<\lambda_1<\lambda_0<\lambda_2$ such that also $u(\varepsilon_1,p_1,\lambda_2 z_1)\ll \phi(\varepsilon_1,p_1)\,\lambda_1 z_1$ and an $n_0>0$ such that $u(\varepsilon_1,p_0{\cdot}t_n,\lambda_2 u(t_n))\ll \phi(\varepsilon_1,p_0{\cdot}t_n)\,\lambda_1 u(t_n)=\lambda_1 u(t_n+\varepsilon_1)$ for $n\geq n_0$. Now, since $\lambda_0<\lambda_2$ and $\lambda_0=\lim_{n\to \infty}\lambda(t_n)$, there exists an $n_1\geq n_0$ such that $b(p_0{\cdot}t_n)\leq \lambda_2 u(t_n)$ for $n\geq n_1$. Therefore, $b(p_0{\cdot}(t_n+\varepsilon_1))=u(\varepsilon_1,p_0{\cdot}t_n,b(p_0{\cdot}t_n))\leq u(\varepsilon_1,p_0{\cdot}t_n,\lambda_2 u(t_n))\leq \lambda_1 u(t_n+\varepsilon_1)$ for $n\geq n_1$, but this means that $\lambda(t_n+\varepsilon_1)\leq \lambda_1<\lambda_0$ for $n\geq n_1$, which is absurd.
\par
(ii) Fixed a $z_0\gg 0$ such that $b(p_0)<z_0$,  let us see that $\lim_{t\to\infty} u(t,p_0,z_0)-b(p_0{\cdot}t)=0$. As in the proof of~(i), we now consider for $t\geq 0$,
\begin{equation}\label{(ii)}
\lambda(t)=\inf\{\lambda\geq 1\mid u(t,p_0,z_0)\leq \lambda\,b(p_0{\cdot}t)\}\,
\end{equation}
which satisfies $\lambda(0)>1$ and once more it is a nonincreasing map on $[0,\infty)$: recall that $b(p_0{\cdot}t)=u(t)$ is a solution of the linear problem, as it has been proved in~(i). Then, let $\lambda_0=\lim_{t\to \infty}\lambda(t)\geq 1$. If $\lambda_0=1$, then fixed any $\varepsilon>0$ there exists a $t_\varepsilon>0$ such that  $\lambda(t)\leq 1+\varepsilon$ for $t\geq t_\varepsilon$, and then
$b(p_0{\cdot}t)\leq u(t,p_0,z_0)\leq (1+\varepsilon)\,b(p_0{\cdot}t)$ for $t\geq t_\varepsilon$. Since $b$ is bounded in norm, this implies that $\lim_{t\to\infty} u(t,p_0,z_0)-b(p_0{\cdot}t)=0$. So, to finish the proof, argue by contradiction and assume that $\lambda_0>1$. Then, we reproduce the same arguments as before in~(i), this time using that $l_{\rm s}^+:=\limsup_{t\to\infty} \|u(t)\|=\limsup_{t\to\infty} \|b(p_0{\cdot}t)\|=r_0$, to get a contradiction.
\par
(iii) If $z\in A(p_0)\subset [-b(p_0),b(p_0)]$, the orbits $(p_0{\cdot}t,b(p_0{\cdot}t))$ and  $(p_0{\cdot}t,u(t,p_0,z))$ lie,  by~(i), in the linear zone of the problem. That is, they provide entire bounded trajectories for the linear skew-product semiflow $\tau_L$, and by Proposition~4.11~(i) in~\cite{calaobsa}, they lie inside the principal bundle. That is to say, $A(p_0)\subset X_1(p_0)$. It is immediate to check that then $A(p_0)=\{\beta\,e(p_0)\mid |\beta|\leq \eta\}$ for $\eta>0$ such that $b(p_0)=\eta\,e(p_0)$. The proof is finished.
\end{proof}
We are now in a position to prove that the pullback attractor $\{A(p{\cdot}t)\}_{t\in  \R}$ is also a forwards attractor for the process $S_p(\cdot,\cdot)$ for $p \in P_{\rm{f}}\cap P_{\rm{r}}$. Note that when $\nu(P_{\rm{f}})=1$, this means that the pullback attractor is a forwards attractor for (at least) all the processes over the invariant set of full measure  $P_{\rm{f}}\cap P_{\rm{r}}$, thus extending Theorem~31 in Caraballo et al.~\cite{caloNonl} in a sublinear ODEs setting, to our PDEs problems with no sublinear assumption. In fact in the sublinear case more can be said: see Section~\ref{sec-sublineal}.
\begin{teor}\label{teor-atraccion forward}
Let $h\in\mathcal{U}(P\times\bar U)$ and let $p_0\in P_{\rm{f}}\cap P_{\rm{r}}$.
Then,
\[
\lim_{t\to\infty} {\rm dist}(u(t,p_0,B),A(p_0{\cdot}t))=0 \quad\text{for any bounded set}\;\,B\subset X.
\]
\end{teor}
\begin{proof}
Fix a bounded set $B\subset X$. By the definition of the Hausdorff semidistance,  it is enough to see that given any $0<\varepsilon<r_0$ we can find a $t_*>0$ such that for any $t\geq t_*$ and for any $u(t,p_0,z)\in u(t,p_0,B)$ ($z\in B$) we can take an appropriate $a(t,z) \in A(p_0{\cdot}t)$ such that $\|u(t,p_0,z)-a(t,z)\|\leq \varepsilon$ (see Remark~\ref{nota-forward atr} in the Dirichlet case). In order to prove this, we make a series of previous helpful assertions.
\begin{itemize}
\item[(a1)] There is a map $\lambda(t)\geq 1$ for $t\geq 1$, with $\lim_{t\to\infty} \lambda(t)=1$ such that $u(t,p_0,B)\subset [-\lambda(t)\,b(p_0{\cdot}t),\lambda(t)\,b(p_0{\cdot}t)]$ for $t\geq 1$.
\end{itemize}
To see it,  note that given a bounded set $B\subset X$, $\tau_1(P\times B)$ is relatively compact in $P\times X^\gamma$ and we can take a $z_0\gg 0$ such that $b(p_0{\cdot}1)\leq z_0$ and $-z_0\leq u(1,p,z)\leq z_0$ for any $p\in P$ and $z\in B$. By monotonicity,  $u(t,p_0{\cdot}1,-z_0)\leq u(t+1,p_0,z)\leq u(t,p_0{\cdot}1,z_0)$ for $z\in B$ and $t\geq 0$, and then apply Theorem~\ref{teor-limsup}~(ii) to $p_0{\cdot}1\in P_{\rm{f}}\cap P_{\rm{r}}$.
\begin{itemize}
\item[(a2)] There is a $\rho>0$ such that, if $\|z\|_\gamma\leq \rho$, then $-e(p)\ll z \ll e(p)$ for $p\in P$.
\end{itemize}
This is because $e:P\to \Int X^\gamma_+$ is continuous, $P$ is compact and $P\times\{0\}\subset \{(p,z)\in P\times X\mid -e(p)\ll z\ll e(p)\}$. The next assertion is well-known for linear skew-product semiflows, and the next one follows from Theorem~\ref{teor-limsup}~(i) and~(iii).
\begin{itemize}
\item[(a3)] $\|\phi(t,p)\|\leq M_0\,e^{c_0t}$ for $t\geq 0$ and $p\in P$, for certain $M_0>0$ and $c_0\in\R$.
\item[(a4)] $b(p_0{\cdot}t)=\eta\,c(t,p_0)\,e(p_0{\cdot}t)$ for $t\geq 0$.
\end{itemize}
\par
Now, for any $t\geq 1$, $z\in B$, $u(t,p_0,z)\in X^\gamma$ and we can write
\[
u(t,p_0,z)=\alpha(t,z)\,b(p_0{\cdot}t)+w(t,z)\in X_1(p_0{\cdot}t)\oplus X_2(p_0{\cdot}t)\,.
\]
Using (a1) we deduce that $(-\lambda(t)-\alpha(t,z))\,b(p_0{\cdot}t)\leq w(t,z)\leq (\lambda(t)-\alpha(t,z))\,b(p_0{\cdot}t)$. Then, it cannot be $-\lambda(t)-\alpha(t,z)> 0$, since then it would be $0\ll  w(t,z)$, but   property (3) of the continuous separation precludes this fact. Analogously, it cannot be $\lambda(t)-\alpha(t,z)< 0$, since then it would be $w(t,z)\ll 0$, and neither can that be. Therefore:
\begin{itemize}
\item[(a5)] $-\lambda(t)\leq \alpha(t,z)\leq \lambda(t)$ for any $t\geq 1$ and $z\in B$.
\end{itemize}
And then it is easy to deduce that:
\begin{itemize}
\item[(a6)] $-2\,\lambda(t)\,b(p_0{\cdot}t)\leq w(t,z)\leq 2\,\lambda(t)\,b(p_0{\cdot}t)$ for any $t\geq 1$ and $z\in B$.
\end{itemize}
\par
From here on, note that some technical details are unnecessary in the Neumann and Robin cases, where only the sup-norm appears. However, we try to unify the writing as much as possible, including the Dirichlet case.   Thinking of the latter case, we first consider $\wit e=\sup_{p\in P}\|e(p)\|_\alpha<\infty$ and we take $E_0=\max(1,\wit e)$; second, for $\wit c=\sup_{p\in P}\|\phi(1,p)\|_{\mathcal{L}(X,X^\alpha)}$ we take $C_0=\max(1,\wit c)$; and third, for $\wit d=\sup_{p\in P} 1/c(1,p)$ we take $D_0=\max(1,\wit d\,)$. Now, by Theorem~\ref{teor-limsup}~(i), $\sup_{t\geq 0} \|b(p_0{\cdot}t)\|=r_0$, and associated to $\varepsilon_0=\varepsilon/(2\,r_0)$ ($\varepsilon_0<1$) we can take a $T>1$ sufficiently big so that
\[
\frac{4\,M\,E_0\,C_0\,D_0\,e^{-\delta (T-1)}}{\rho}<\varepsilon_0\,,
\]
where $M>0$ and $\delta>0$ are the constants given in property (5) of the continuous separation (see Remark~\ref{nota-dirichlet} in the Dirichlet case), and $\rho>0$ is the one in (a2).
Given this $T>1$, assuming in the worst case that  $c_0$ in (a3) is positive, we can take a $\delta_0>0$ small enough so that $M_0\,e^{c_0 T}\,\delta_0<r_0$ in such a way that, if we start with initial conditions with $\|z_0\|\leq \delta_0$, then the linear mild solutions for $t\in [0,T]$ remain in the linear zone of the problems providing  solutions of the nonlinear problem too. More precisely, by (a3), $\|\phi(t,p)\,z_0\|\leq M_0\,e^{c_0t}\|z_0\|< r_0$ for $t\in [0,T]$ and $p\in P$.
\par
Since we have (a1) and by Proposition~\ref{prop-norma b},  $\liminf_{t\to\infty} \|b(p_0{\cdot}t)\|=0$,   we can take a $t_0>2$ such that $\lambda(t)\leq 1+\varepsilon_0$ for any $t\geq t_0-1$ and $\|u(t_0-1,p_0,z)\|\leq \delta_0$ for any $z\in B$. Then, as seen in the previous paragraph, for $t\in [0,T]$ and $z\in B$, $u(t_0-1+t,p_0,z)$ remains in the linear zone. Some technical details which appear in the Dirichlet case are the reason why we start at time $t_0-1$, but in fact we pay attention at the dynamics from time $t_0$ on, so that it is convenient to write  $t_1=T-1>0$ and then, for any $z\in B$, recalling that $b(p_0{\cdot}t)$ lies in the linear zone by Theorem~\ref{teor-limsup}~(i),
\begin{equation}\label{formula}
\begin{split}
u(t_0+t_1,p_0,z)&=u(t_1,p_0{\cdot}t_0,u(t_0,p_0,z))=\phi(t_1,p_0{\cdot}t_0)\,u(t_0,p_0,z)
\\&= \phi(t_1,p_0{\cdot}t_0)\,\alpha(t_0,z)\,b(p_0{\cdot}t_0)+ \phi(t_1,p_0{\cdot}t_0)\,w(t_0,z)
\\&= \alpha(t_0,z)\,\phi(t_1,p_0{\cdot}t_0)\,b(p_0{\cdot}t_0)+ w(t_0+t_1,z)
\\&=\alpha(t_0,z)\,b(p_0{\cdot}(t_0+t_1)) + w(t_0+t_1,z)\,.
\end{split}
\end{equation}
\par
Now,  $\|w(t_0+t_1,z)\|_\gamma= \|\phi(t_1,p_0{\cdot}t_0)\,w(t_0,z)\|_\gamma$
and by property (5) of the continuous separation, $\|w(t_0+t_1,z)\|_\gamma\leq M\, e^{-\delta t_1} \|\phi(t_1,p_0{\cdot}t_0)\,e(p_0{\cdot}t_0)\|_\gamma \,\|w(t_0,z)\|_\gamma$. In the Neumann and Robin cases, using first
(a6) and then (a4), we can bound
\begin{align*}
\|w(t_0+t_1,z)\|&\leq 4\,M\, e^{-\delta t_1} \,c(t_1,p_0{\cdot}t_0)\,\|b(p_0{\cdot}t_0)\|=4\,M\, e^{-\delta t_1} \,c(t_1,p_0{\cdot}t_0)\,\eta\,c(t_0,p_0)
\\& = \frac{ 4\,M\, e^{-\delta t_1}}{\rho} \,c(t_0+t_1,p_0)\,\eta \,\rho \leq \varepsilon_0 \,c(t_0+t_1,p_0)\,\eta \,\rho\,,
\end{align*}
whereas in the Dirichlet case $\|\,{\cdot}\,\|_\alpha$ is not monotone, and the bound is more delicate. First,
$\|\phi(t_1,p_0{\cdot}t_0)\,e(p_0{\cdot}t_0)\|_\alpha\leq c(t_1,p_0{\cdot}t_0)\,E_0$. Second, by the choice of $t_0$,  $\|w(t_0,z)\|_\alpha=\|\phi(1,p_0{\cdot}(t_0-1))\,w(t_0-1,z)\|_\alpha\leq C_0\,\|w(t_0-1,z)\|$, and once we have the sup-norm, we can apply (a6) and (a4) to get $\|w(t_0-1,z)\|\leq 4\,\|b(p_0{\cdot}(t_0-1))\|=4\,\eta\,c(t_0-1,p_0)\leq 4\,\eta\,D_0\,c(1,p_0{\cdot}(t_0-1))\,c(t_0-1,p_0)=4\,\eta\,D_0\,c(t_0,p_0)$. Then,
\[
\|w(t_0+t_1,z)\|_\alpha\leq \frac{ 4\,M\,E_0\,C_0\,D_0\, e^{-\delta t_1}}{\rho} \,c(t_0+t_1,p_0)\,\eta \,\rho \leq \varepsilon_0 \,c(t_0+t_1,p_0)\,\eta \,\rho\,.
\]
In both cases $\|w(t_0+t_1,z)/(\varepsilon_0 \,c(t_0+t_1,p_0)\,\eta)\|_\gamma\leq \rho$
and, by (a2),  for any $z\in B$,
\[
-\varepsilon_0 \,c(t_0+t_1,p_0)\,\eta\,e(p_0{\cdot}(t_0+t_1))\leq  w(t_0+t_1,z) \leq  \varepsilon_0 \,c(t_0+t_1,p_0)\,\eta\,e(p_0{\cdot}(t_0+t_1))\,.
\]
Once more, by (a4), $-\varepsilon_0\,b(p_0{\cdot}(t_0+t_1))\leq  w(t_0+t_1,z) \leq \varepsilon_0\,b(p_0{\cdot}(t_0+t_1))$ for any $z\in B$, and coming back to the expression~\eqref{formula}  of $u(t_0+t_1,p_0,z)$, we get that
\[
(\alpha(t_0,z)-\varepsilon_0)\,b(p_0{\cdot}(t_0+t_1))\leq u(t_0+t_1,p_0,z)\leq (\alpha(t_0,z)+\varepsilon_0)\,b(p_0{\cdot}(t_0+t_1))\,.
\]
Now  we study how the dynamics evolves for $t\geq t_*$ for $t_*=t_0+t_1$. Recalling relation (a5), we distinguish three cases covering all the possible situations.
\par
{\it Case 1}: $-1\leq \alpha(t_0,z)-\varepsilon_0\leq \alpha(t_0,z)+\varepsilon_0\leq 1$. Then, everything remains in the linear zone of the problems, so that for any $s\geq 0$, by monotonicity,
\[
(\alpha(t_0,z)-\varepsilon_0)\,b(p_0{\cdot}(t_*+s))\leq u(t_*+s,p_0,z)\leq (\alpha(t_0,z)+\varepsilon_0)\,b(p_0{\cdot}(t_*+s))\,,
\]
so that for each $t=t_*+s\geq t_*$ and $z\in B$ we can take $a(t,z)=\alpha(t_0,z)\,b(p_0{\cdot}t)\in A(p_0{\cdot}t)$ and $\|u(t,p_0,z)-a(t,p)\|\leq \varepsilon_0\,\|b(p_0{\cdot}t)\|\leq \varepsilon$ and we are done in this case.
\par
{\it Case 2}: $1 -\varepsilon_0\leq \alpha(t_0,z)\leq \lambda(t_0)\leq 1+\varepsilon_0$, so that $-1<1-2\,\varepsilon_0\leq \alpha(t_0,z)-\varepsilon_0\leq 1$. Then, for $s\geq 0$, $
(1-2\,\varepsilon_0)\,b(p_0{\cdot}(t_*+s))\leq u(t_*+s,p_0,z)\leq\lambda(t_*+s)\,b(p_0{\cdot}(t_*+s))$,
and $\lambda(t_*+s)\leq 1+\varepsilon_0$ for any $s\geq 0$. Thus, for each $t=t_*+s\geq t_*$ and $z\in B$ we can take $a(t,z)=b(p_0{\cdot}t)\in A(p_0{\cdot}t)$ and $\|u(t,p_0,z)-a(t,p)\|\leq 2\,\varepsilon_0\,\|b(p_0{\cdot}t)\|\leq \varepsilon$.
\par
{\it Case 3}: $-1-\varepsilon_0\leq -\lambda(t_0)\leq \alpha(t_0,z)\leq -1+\varepsilon_0$. This case is analogous to the previous case. Only take $a(t,z)=-b(p_0{\cdot}t)\in A(p_0{\cdot}t)$ to get that $\|u(t,p_0,z)-a(t,p)\|\leq 2\,\varepsilon_0\,\|b(p_0{\cdot}t)\|\leq \varepsilon$.
The proof is finished.
\end{proof}
This is a direct consequence of Proposition~\ref{prop-forward attractor} and Theorem~\ref{teor-atraccion forward}.
\begin{coro}\label{coro-Kozlov}
Assume that $h\in\mathcal{U}(P\times\bar U)$ is such that every point $p\in P$ is recurrent at $\pm \infty$, i.e., $P=P_{\rm{r}}$. Then, for every $p\in P_{\rm{s}}$ the process $S_p(\cdot,\cdot)$ has no forwards attractor, whereas
$\{A(p{\cdot}t)\}_{t\in  \R}$ is its forwards attractor for every  $p \in P_{\rm{f}}$.
\end{coro}
\begin{eje}\label{eje-kozlov}
In the quasi-periodic case with $P=\T^n$ for some $n\geq 2$, if $h\in C^k(\T^n)\cap C_0(\T^n)$ for $k$ big enough, then the smooth 1-dim cocycle $c(t,p)=\exp\int_0^t h(p{\cdot}s)\,ds$ satisfies that every $p\in \T^n$ is recurrent. This follows from results by  Kozlov~\cite{kozl},  Konyagin~\cite{kony} y Moshchevitin~\cite{mosh}. It is not difficult to check that the linear-dissipative family  of problems   with spatially homogeneous linear part
\begin{equation*}
\left\{\begin{array}{l} \des\frac{\partial y}{\partial t}  =
 \Delta \, y+(\gamma_0+h(p{\cdot}t))\,y+g(p{\cdot}t,x,y)\,,\quad t>0\,,\;\,x\in U, \;\, \text{for each}\; p\in \T^n,
  \\[.2cm]
By:=\alpha(x)\,y+\delta\,\des\frac{\partial y}{\partial n} =0\,,\quad  t>0\,,\;\,x\in \partial U,\,
\end{array}\right.
\end{equation*}
where $\gamma_0$ is the first eigenvalue of the boundary value  problem~\eqref{bvp} (with the corresponding boundary conditions), has $c(t,p)$ as its associated 1-dim cocycle.
\end{eje}
Some of the arguments used in the proof of Theorem~\ref{teor-atraccion forward} permit us to prove that the sections of the attractor are always contained in the union of the positive and negative cones of $X^\gamma$. This is one of the main results in the paper.
\begin{teor}\label{teor-A en cono positivo}
For any $h\in C_0(P\times \bar U)$ the global attractor $\A$ satisfies:
\[
\A\subset P\times (\Int X^\gamma_+\cup \Int X^\gamma_-\cup \{0\})\,.
\]
\end{teor}
\begin{proof}
First of all, recall that $\A\subset P\times X^\gamma$. The result is well-known if $h\in B(P\times \bar U)$: see Theorem~\ref{teor-estr atractor caso b}. So, assume that $h\in\mathcal{U}(P\times\bar U)$. If $p\in P_{\rm{s}}$, $b(p)=0$ and $A(p)=\{0\}$.
 For $p_0\in P_{\rm{f}}$, let us take a pair $(p_0,z)\in \A$ with $z\not= 0$. If we assume that there exists a $t_0<0$ such that for any $t\leq t_0$, $\|u(t,p_0,z)\|\leq r_0$, then $\{(p_0{\cdot}t,u(t,p_0,z))\mid t\leq t_0\}\subset \A$ is a bounded  semiorbit for the linear skew-product semiflow $\tau_L$, which can be continued to the whole line $\R$.  Proposition~4.11~(i) in~\cite{calaobsa} implies that $u(t,p_0,z)\in X_1(p_0{\cdot}t)$ for any $t\leq t_0$. Thus, either $u(t_0,p_0,z)\gg 0$ or $u(t_0,p_0,z)\ll 0$, and by the monotonicity of the semiflow either $z\gg 0$ or $z\ll 0$.
\par
Now, argue assuming that such a $t_0<0$ does not exist. By Proposition~\ref{prop-norma b}, $\liminf_{t\to -\infty}\|b(p_0{\cdot}t)\|=0$ and therefore, also $\liminf_{t\to -\infty}\|u(t,p_0,z)\|=0$, so that $\|u(t,p_0,z)\|$ necessarily crosses the threshold $r_0$ infinitely many times as $t\to -\infty$.
\par
At this point, first, we recover some assertions from the proof of Theorem~\ref{teor-atraccion forward}; more precisely, let $\rho>0$ and $M_0>0$, $c_0\in\R$ be the constants given in (a2) and (a3), respectively, and assume in the worst case that $c_0>0$.  Second, let $M,\,\delta>0$ be the constants involved in property (5) of the continuous separation (see Remark~\ref{nota-dirichlet} in the Dirichlet case). Then, since we are assuming that $\Sigma_{\text{pr}}=\{0\}$,  there is an exponential dichotomy with full stable subspace for the 1-dim semiflow $e^{-\delta t}\,\phi(t,p) |_{X_1}$, which in particular implies that $\lim_{t\to\infty} e^{-\delta t}c(t,p)=0$ uniformly for $p\in P$ (actually exponentially fast). Third, take a $\lambda_0>0$ such that $b(p)\leq \lambda_0\,e(p)$ for any $p\in P$. Then, fixed $\varepsilon>0$ such that $2\,\varepsilon<r_0$, there exists a $T>1$ such that
\[
\frac{2\,\lambda_0\,E_0\,C_0\,M\,e^{-\delta t}c(t,p)}{\rho}\leq \varepsilon\quad\text{for }t\geq T-1,\;p\in P,
\]
where $E_0=\max(1,\wit e)$ for $\wit e=\sup_{p\in P}\|e(p)\|_\alpha<\infty$ and $C_0=\max(1,\wit c)$ for $\wit c=\sup_{p\in P}\|\phi(1,p)\|_{\mathcal{L}(X,X^\alpha)}$ are needed in the Dirichlet case.
\par
Last, associated to $T>1$, we can fix a $\delta_0>0$ small enough so that $M_0\,e^{c_0 T}\delta_0<r_0$, guaranteeing that, whenever at some $s\in\R$, $\|u(s,p_0,z)\|\leq \delta_0$, then the solution of the linear problem starting at $(p_0{\cdot}s,u(s,p_0,z))$ stays strictly in the linear zone of the problems for an interval of length at least $T$. 
\par
Since $\liminf_{t\to -\infty}\|u(t,p_0,z)\|=0$ and $\|u(t,p_0,z)\|$ crosses the threshold $r_0$ infinitely many times as $t\to -\infty$, it is clear that we can take a $t_0<0$ such that $\|u(t_0-1,p_0,z)\|\leq \delta_0$ and the minimum $T_1> T$ such that $\|u(t_0-1+T_1,p_0,z)\|= r_0$, with $t_0-1+T_1<0$. Let us write for $t\in [-1,t_1]$ for $t_1=T_1-1>0$,
\[
u(t_0+t,p_0,z)=\alpha(t)\,e(p_0{\cdot}(t_0+t))+ w(t) \in X_1(p_0{\cdot}(t_0+t))\oplus X_2(p_0{\cdot}(t_0+t))\,.
\]
As in the proof of Theorem~\ref{teor-atraccion forward}, since for any $t$, $-b(p_0{\cdot}(t_0+t))\leq u(t_0+t,p_0,z)\leq b(p_0{\cdot}(t_0+t))$, it is easy to check that $-\lambda_0\leq \alpha(t)\leq \lambda_0$ and $-2\,\lambda_0\,e(p_0{\cdot}(t_0+t))\leq w(t)\leq 2\,\lambda_0\,e(p_0{\cdot}(t_0+t))$ for $t\in [-1,t_1]$. Now, we can solve the nonlinear problem starting at $(p_0{\cdot}(t_0-1),u(t_0-1,p_0,z))$, whose solution coincides with the solution of the linear problem at least for $t\in [0,T_1]$. In particular, $u(t_0+t_1,p_0,z)=\phi(t_1,p_0{\cdot}t_0)\,u(t_0,p_0,z)$. Then, applying (5) in the description of the continuous separation, $\|w(t_1)\|_\gamma=\|\phi(t_1,p_0{\cdot}t_0)\,w(0)\|_\gamma\leq M\,e^{-\delta t_1}\|\phi(t_1,p_0{\cdot}t_0)\,e(p_0{\cdot}t_0)\|_\gamma\,\|w(0)\|_\gamma$
$\leq 2\,\lambda_0\, M\,E_0\,C_0\,e^{-\delta t_1} c(t_1,p_0{\cdot}t_0)\leq \varepsilon\,\rho$. Note that in the Dirichlet case we bound $\|w(0)\|_\alpha=\|\phi(1,p_0{\cdot}(t_0-1))\,w(-1)\|_\alpha\leq C_0\,\|w(-1)\|\leq C_0\,2\,\lambda_0$.
Thus,  we have that $\|w(t_1)/\varepsilon\|_\gamma\leq \rho$, and then by assertion (a2) in the repeatedly mentioned proof, $-\varepsilon\,e(p_0{\cdot}(t_0+t_1))\leq w(t_1)\leq \varepsilon\,e(p_0{\cdot}(t_0+t_1))$, so that
\[
(\alpha(t_1)-\varepsilon)\,e(p_0{\cdot}(t_0+t_1))\leq u(t_0+t_1,p_0,z)\leq (\alpha(t_1)+\varepsilon)\,e(p_0{\cdot}(t_0+t_1))\,.
\]
Since $\|u(t_0+t_1,p_0,z)\|= r_0$, it must be either $\alpha(t_1)\geq r_0-\varepsilon$ or $\alpha(t_1)\leq -r_0+\varepsilon$. In the first case, $u(t_0+t_1,p_0,z)=\alpha(t_1)\,e(p_0{\cdot}(t_0+t_1))+ w(t_1)\geq (r_0-\varepsilon)\,e(p_0{\cdot}(t_0+t_1))-\varepsilon\,e(p_0{\cdot}(t_0+t_1))=(r_0-2\,\varepsilon)\,e(p_0{\cdot}(t_0+t_1))\gg 0$ and moving along the orbit, $z\gg 0$. In the second case, $u(t_0+t_1,p_0,z)\leq (-r_0+2\,\varepsilon)\,e(p_0{\cdot}(t_0+t_1))\ll 0$ and then $z\ll 0$. The proof is finished.
\end{proof}
We finish this section with a result on chaotic dynamics in some of the pullback attractors. We first introduce the concept of Li-Yorke chaos of the pullback attractor for a process $S(\cdot,\cdot)$.  Note that this is a new notion, up to our knowledge, since the classical concept of chaos in the  sense of Li and Yorke~\cite{liyo} is given for flows on compact metric spaces.
\begin{defi}\label{defi-li yorke}
Let $\{S(t,s)\mid t\geq s\}\subset C(X)$ be a process and assume that it has a pullback attractor $\{A(t)\}_{t\in \R}$.
\par
(i) A pair $\{z_1,z_2\}\subset A(s)$ ($s\in\R$) is called a {\it Li-Yorke pair\/} if
\[
\liminf_{(t\geq s)\,t\to \infty} \|S(t,s)\,z_2-S(t,s)\,z_1\|=0\;\;\text{and}\;\, \limsup_{(t\geq s)\,t\to \infty} \|S(t,s)\,z_2-S(t,s)\,z_1\|>0\,.
\]
\par
(ii) A set $D\subseteq A(s)$ ($s\in\R$) is said to be {\it scrambled\/} if every pair $\{z_1,z_2\}\subset D$ with $z_1\not=z_2$ is a Li-Yorke pair.
\par
(iii) The pullback attractor is said to be {\it chaotic in the sense of Li-Yorke\/} if there exists a family $\{D(t)\}_{t\in\R}$ with $D(t)\subseteq A(t)$ and
\begin{itemize}
\item[(1)] $S(t,s)\,D(s)=D(t)$ for $t\geq s$;
\item[(2)] $D(t)$ is an uncountable scrambled set in $X$ for $t\in \R$.
\end{itemize}
\end{defi}
We remark that for our processes  $S_p(\cdot,\cdot)$ for $p\in P$, which are involved with the skew-product semiflow $\tau$, to get the chaotic behaviour of the pullback attractor in the sense of Li-Yorke, it is enough that for $t=0$ there exists an uncountable scrambled set $D\subset A(p)$. Then, for $s>0$ we take $D(s)=S_p(s,0)\,D$ and for $s<0$,  $D(s)=S_p(0,s)^{-1}\,D$ (we can do that because the semiflow is a flow inside the attractor $\A$) and it is easy to check that (1) and (2) hold.
\begin{teor}
Let $h\in\mathcal{U}(P\times\bar U)$. Then, for any $p\in P_{\rm{f}}\cap P_{\rm{r}}$  the pullback attractor $\{A(p{\cdot}t)\}_{t\in  \R}$ for the process $S_p(\cdot,\cdot)$ is chaotic in the sense of Li-Yorke.
\end{teor}
\begin{proof}
According to the previous remark, it suffices to prove that the set $A(p)$ itself is uncountable and scrambled. Since $p\in P_{\rm{f}}$, the set $A(p)$ is uncountable. Now, Theorem~\ref{teor-limsup}~(iii) asserts that $A(p)\subset X_1(p)$. Then, taking $z_1, z_2\in A(p)$ with $z_1\not= z_2$, we can write $z_1=\lambda_1\,b(p)$ and $z_2=\lambda_2\,b(p)$ for real $\lambda_1\not= \lambda_2$. Here $S_p(t,0)\,z_2-S_p(t,0)\,z_1=u(t,p,z_2)-u(t,p,z_1)$ for $t\geq 0$. Since by Theorem~\ref{teor-limsup}~(i), $b(p{\cdot}t)$ remains in the linear zone of the problems, so do the orbits of $(p,z_1)$ and $(p,z_2)$ and we have that $\|u(t,p,z_2)-u(t,p,z_1)\|=|\lambda_2- \lambda_1|\,\|b(p{\cdot}t)\|$ for $t\geq 0$. To conclude the proof, just recall that  by Proposition~\ref{prop-norma b}, $\liminf_{t\to\infty}\|b(p{\cdot}t)\|=0$, and  by Theorem~\ref{teor-limsup}~(i), $\limsup_{t\to\infty}\|b(p{\cdot}t)\|=r_0$, so that $(p,z_1)$ and $(p,z_2)$ form a Li-Yorke pair, and we are done.
\end{proof}
\section{Linear-dissipative problems with $\lambda_P=0$. The sublinear case}\label{sec-sublineal}\noindent
In this section we can go further in our dynamical study of the linear-dissipative problems~\eqref{pdefamilynl}, giving more details on the inner structure of the attractor and getting forwards attraction additionally for the processes $S_p(\cdot,\cdot)$ for $p\in P_{\rm{f}}$ with $\limsup_{t\to \infty}c(t,p)=\infty$, by assuming a strict sublinearity condition on the nonlinear term $g(p,x,y)$ for $y>r_0$. More precisely, we assume conditions (c1)-(c5) plus
\begin{itemize}
\item[(c6)] $g(t,x,\lambda\,y)< \lambda\,g(t,x,y)$  for $p\in P$,
$x\in\bar U$, $y>r_0$ and $\lambda\geq 1$.
\end{itemize}
With this extra condition, a standard argument of comparison of solutions shows that the skew-product  semiflow $\tau$ is sublinear too, that is,  $u(t,p,\lambda\, z)\leq \lambda\, u(t,p,z)$ for  $\lambda\geq 1$, $p\in P$, $z\geq 0$ and $t\geq 0$. Note that the sublinear case includes the nonautonomous linear-dissipative versions of the Chafee-Infante equation (see Chafee and Infante~\cite{chin} and Carvalho et al.~\cite{calaro12}), as well as of the Fisher equations (see Shen and Yi~\cite{shyi2}).
\par
Proposition~\ref{prop-norma b} asserts that for $p_0\in P_{\rm{f}}$,
$\limsup_{t\to -\infty} \|b(p_0{\cdot}t)\|\geq r_0$. It turns out that in the sublinear case, this is exclusive for the boundary maps  of the attractor. This fact has strong consequences on the structure of $\A$; more precisely,  all the  trajectories strictly inside the attractor  eventually enter, going backwards in time, the principal bundle~\eqref{principal bundle} of the continuous separation. Do not forget that the attractor $\A\subset P\times (\Int X^\gamma_+\cup \Int X^\gamma_-\cup \{0\})$. We concentrate on the behaviour in the positive cone, but analogous results hold in the negative cone.
\begin{teor}\label{teor-A sublineal}
Let $h\in\mathcal{U}(P\times\bar U)$ and assume that $g(p,x,y)$ satisfies $\rm{(c1)}$-$\rm{(c6)}$. Let $p_0\in P_{\rm{f}}$.  Then:
\begin{itemize}
\item[(i)] If $(p_0,z_0)\in \A$ with $0\ll z_0 < b(p_0)$, then $\des\limsup_{t\to -\infty} \|u(t,p_0,z_0)\|<r_0$. Thus, there exists a $t_0<0$ such that $u(t,p_0,z_0)\in X_1(p_0{\cdot}t)$ for any $t\leq t_0$.
\item[(ii)] If $\des\limsup_{t\to \infty} c(t,p_0)=\infty$, then:
\begin{itemize}
\item[(ii.1)] For any $0< z_0<b(p_0)$ it holds that $\limsup_{t\to \infty} \|u(t,p_0,z_0)\|\geq r_0$ and $\des\lim_{t\to\infty} b(p_0{\cdot}t)-u(t,p_0,z_0)=0$.
\item[(ii.2)] If $z_0\gg 0$ is such that $b(p_0)\leq z_0$, then $\des\lim_{t\to\infty} u(t,p_0,z_0)-b(p_0{\cdot}t)=0$.
\end{itemize}
\end{itemize}
\end{teor}
\begin{proof}
(i) Recall that we have bounded entire orbits inside $\A$ and  $0\ll z_0 < b(p_0)$, so that it must be  $0\ll u(t,p_0,z_0)\leq b(p_0{\cdot}t)$ for any $t\in\R$ by Theorem~\ref{teor-A en cono positivo}. Although the main ideas are those  in the proof of Theorem~\ref{teor-limsup}~(i), we include some details for the sake of completeness. Let us define for each $t\in\R$,
\begin{equation}\label{lambda(t)}
\lambda(t)=\inf\{\lambda\geq 1\mid b(p_0{\cdot}t)\leq \lambda\, u(t,p_0,z_0)\}\,.
\end{equation}
By its definition, $\lambda(0)> 1$ and using the cocycle identity and the sublinearity of the semiflow in the positive cone, it is easy to check that $\lambda(t)$ is nonincreasing on $\R$. At this point, we assume by contradiction that $l_{\rm s}^-=\limsup_{t\to -\infty} \|u(t,p_0,z_0)\|\geq r_0$. Then, arguing exactly as in the proof of Theorem~\ref{teor-limsup}~(i) (now it is $u(t,p_0,z_0)$ playing the role of $u(t)$ therein), we get that  it must be $\lim_{t\to-\infty}\lambda(t)=\lambda_0\in(1,\infty)$.
This time we consider (instead of the linearized problems) the auxiliary family of parabolic PDEs given for  each $p\in P$ by
\begin{equation}\label{pde-auxiliar}
\left\{\begin{array}{l} \des\frac{\partial y}{\partial t}  =
 \Delta \, y+h(p{\cdot}t,x)\,y+\lambda_0\,g\big(p{\cdot}t,x,\frac{y}{\lambda_0}\big)\,,\quad t>0\,,\;\,x\in U,
  \\[.2cm]
By:=\alpha(x)\,y+\delta\,\des\frac{\partial y}{\partial n} =0\,,\quad  t>0\,,\;\,x\in \partial U,\,
\end{array}\right.
\end{equation}
and denote the mild solutions of the associated ACPs by $v(t,p,z)$. With condition (c6) on $g$, by comparison we obtain that $u(t,p,z)\leq v(t,p,z)$ for $p\in P$, $z\geq 0$ and $t\geq 0$, and it is easy to check that $v(t,p, \lambda_0 z)=\lambda_0 u(t,p,z)$ for $p\in P$, $z\geq 0$, $t\geq 0$.
\par
Then, 
taking $\delta_0, \delta_1>0$ as in the mentioned proof,  as $l_{\rm s}^- \geq r_0$, there is a sequence $(t_n)_n\downarrow -\infty$ such that $\|u(t_n,p_0,z_0)\|\geq r_0-\delta_1$, $n\geq 1$ and $(p_0{\cdot}t_n, u(t_n,p_0,z_0))\to (p_1,z_1)\in \A$ as $n\to \infty$. In particular $\|\lambda_0z_1\|> r_0+\delta_0$. This time comparing the solutions of the nonlinear problem and the auxiliary problem~\eqref{pde-auxiliar},  $u(\varepsilon_1,p_1,\lambda_0 z_1)\ll v(\varepsilon_1,p_1,\lambda_0 z_1)=\lambda_0 u(\varepsilon_1,p_1,z_1)$ for some $\varepsilon_1>0$. By continuity, we can take constants $1<\lambda_1<\lambda_0<\lambda_2$ and an $n_0>0$ so that  $u(\varepsilon_1,p_0{\cdot}t_n,\lambda_2 u(t_n,p_0,z_0))\ll \lambda_1 u(\varepsilon_1,p_0{\cdot}t_n, u(t_n,p_0,z_0))=\lambda_1 u(t_n+\varepsilon_1)$ for $n\geq n_0$, and this leads to a contradiction in the same way as in that proof.
Thus, $l_{\rm s}^-<r_0$. This means that $u(t,p_0,z_0)$ lies in the linear zone of the problem for $t\leq t_0$ for some $t_0<0$, and we can  argue as in the first lines in the proof of Theorem~\ref{teor-A en cono positivo}.
\par
(ii) Let us assume that $\limsup_{t\to \infty} c(t,p_0)=\infty$, which by Proposition \ref{prop-crossing} implies that $l_{\rm s}^+=\limsup_{t\to \infty} \|b(p_0{\cdot}t)\|\geq r_0$.
\par
(ii.1) To see that $\limsup_{t\to \infty} \|u(t,p_0,z_0)\|\geq r_0$ for $0<z_0<b(p_0)$, one just argues as in the proof of Proposition~\ref{prop-crossing}~(ii). Now, we remark that we can assume without loss of generality that $0\ll z_0 <b(p_0)$, so that we can consider the nonincreasing map $\lambda(t)$ in~\eqref{lambda(t)} defined for $t\geq 0$ with $\lambda(0)>1$: just note that if $0< z_0 <b(p_0)$, $\lambda(0)$ might not be well defined, but we can apply the strong monotonicity of the semiflow to get $0\ll u(t,p_0,z_0)\ll b(p_0{\cdot}t)$ for any $t>0$ and then look at $\lambda(t)$ defined on an interval $[\varepsilon_0,\infty)$ for some $\varepsilon_0>0$, with $\lambda(\varepsilon_0)>1$. This time take $\lambda_0=\lim_{t\to\infty} \lambda(t)\geq 1$. It is immediate to check that if $\lambda_0=1$ then $\lim_{t\to\infty} b(p_0{\cdot}t)-u(t,p_0,z_0)=0$. So, argue by contradiction and assume that $\lambda_0>1$. Again we consider the auxiliary family of parabolic PDEs given for  each $p\in P$ by~\eqref{pde-auxiliar}, and we reproduce the previous arguments in~(i), with the obvious necessary modifications, in order to get a contradiction.
\par
(ii.2) The proof follows the same lines as the proof of Theorem~\ref{teor-limsup}~(ii). For a fixed $z_0\gg 0$ with $b(p_0)<z_0$, the map $\lambda(t)$ is defined for $t\geq 0$ exactly as in~\eqref{(ii)}. It satisfies that $\lambda(0)>1$ and it is nonincreasing thanks to the sublinearity and the fact that $b$ is an equilibrium for $\tau$. We take $\lambda_0=\lim_{t\to\infty}\lambda(t)\geq 1$. Then, if $\lambda_0=1$ we are done, whereas if we assume that $\lambda_0>1$ we get a contradiction, arguing as in the proof of (i) in the present theorem, this time using that $l_{\rm s}^+\geq r_0$.
\end{proof}
As we pointed out in Remark~\ref{nota-1}~(ii), there is always a nontrivial segment of $X_1(p)$ in the section of the attractor $A(p)$ for any $p\in P_{\rm{f}}$. Besides,  the whole section $A(p)$ is a segment in $X_1(p)$ if $p\in P_{\rm{f}}\cap P_{\rm{r}}$, by Theorem~\ref{teor-limsup}~(iii). In the sublinear case we can give some further details on the structure of $A(p)$ in order to prove the forwards attraction of the pullback attractor for some $p\in P_{\rm{f}}\setminus P_{\rm{r}}$.
\begin{prop}\label{prop-b1}
Let $h\in\mathcal{U}(P\times\bar U)$ and assume $\rm{(c1)}$-$\rm{(c6)}$ on $g(p,x,y)$. Then:
\begin{itemize}
\item[(i)] For $p\in P_{\rm{f}}$, all the pairs in the section $A(p)$ are strongly ordered.
\item[(ii)] For $p\in P_{\rm{f}}$, define
\begin{equation}\label{b1}
b_1(p)=\sup\{z\gg 0\mid (p,z)\in \A \text{ and }\|u(t,p,z)\|\leq r_0 \text{ for } t\leq 0\}\,.
\end{equation}
Then, $b_1(p)\gg 0$, $b_1(p)\in X_1(p)$ and if $(p,z)\in \A$ with $0\leq z\leq b_1(p)$, then also $z\in X_1(p)$. Besides,  $b_1(p{\cdot}t)\leq u(t,p,b_1(p))$ for any $t\geq 0$, $p\in P_{\rm{f}}$ and
\[
b(p)=\lim_{t\to\infty} u(t,p{\cdot}(-t),b_1(p{\cdot}(-t)))\quad \text{for }\, p\in P_{\rm{f}}\,.
\]
\end{itemize}
\end{prop}
\begin{proof}
(i) As usual, we just argue for the intersection of $A(p)$ with the positive cone. If $(p,z)\in \A$ with $z>0$, by Theorem~\ref{teor-A en cono positivo} we already know that $z\gg 0$. And if $z<b(p)$, we just move backwards along the full orbit of $(p,z)$ in the attractor which satisfies $u(t,p,z)<b(p{\cdot}t)$ for $t<0$  and come back forwards to get  $z\ll b(p)$, by the strong monotonicity. Finally, for two distinct points $z_1,z_2\in A(p)$, $0\ll z_1, z_2<b(p)$, the result is a consequence of Theorem~\ref{teor-A sublineal}~(i): just move backwards in time till both orbits enter the principal bundle. Since they are necessarily ordered in the past, $z_1$ and $z_2$ are strongly ordered.
\par
(ii) Let us check that, for the constant $m(p)=\sup_{t\leq 0} c(t,p)$,
\begin{equation}\label{igualdad}
\{z\gg 0\mid (p,z)\in \A \text{ and }\|u(t,p,z)\|\leq r_0 \,\forall\, t\leq 0\}=\Big\{re(p)\mid 0<r\leq \frac{r_0}{m(p)}\Big\}.
\end{equation}
Note that the inclusion $\supseteq$ has been proved in Remark \ref{nota-1} (ii). Conversely, if we take a $z\gg 0$ such that $(p,z)\in \A$ and its past semitrajectory remains in the linear zone,  then  $z\in X_1(p)$, once more
 by Proposition~4.11~(i) in~\cite{calaobsa}. If it were $z=re(p)$ for an $r>r_0/m(p)$, then we would have $r_0<r c(t_0,p)$ for some $t_0\leq 0$, but then $\|\phi(t_0,p)\,re(p)\|=\|r c(t_0,p)\,e(p{\cdot}t_0)\|>r_0$, in contradiction with the choice of $z$. Therefore,  $ 0<r\leq r_0/m(p)$ and the inclusion $\subseteq$ also holds.
\par
As a  consequence of the previous equality,
\begin{equation}\label{b_1(p)}
0\ll b_1(p)=\frac{r_0}{m(p)}\,e(p)\leq r_0\,e(p)\,,
\end{equation}
and if $(p,z)\in \A$ with $0\ll z< b_1(p)$,  by Theorem~\ref{teor-A sublineal}~(i) there is a $t_0<0$ such that $u(t,p,z)\in X_1(p{\cdot}t)$ for $t\leq t_0$. As also $u(t,p,b_1(p))\in X_1(p{\cdot}t)$ for $t\leq t_0$ these two elements must be ordered, and since $z< b_1(p)$, necessarily $u(t,p,z)< u(t,p,b_1(p))$ for $t\leq t_0$. By monotonicity, moving forwards, we conclude that $u(t,p,z)< u(t,p,b_1(p))$ for any $t\leq 0$ and therefore  also the past semitrajectory of $(p,z)$ remains in the linear zone,  and thus $z\in X_1(p)$.
\par
Now, to see that $b_1(p{\cdot}t)\leq u(t,p,b_1(p))$ for $t\geq 0$, let us check that for any $z\gg 0$ such that $(p{\cdot}t,z)\in \A$ and whose negative semiorbit lies in the linear zone of the problems, one has that $z\leq u(t,p,b_1(p))$. Since we have a flow on $\A$ we can write $z=u(t,p,z_0)$ for a certain $(p,z_0)\in\A$ with $z_0\gg 0$. Since the past of  $(p,z_0)$ is part of the past of $(p{\cdot}t,z)$, which is in the linear zone, $z_0\leq b_1(p)$. By monotonicity, $z=u(t,p,z_0)\leq u(t,p,b_1(p))$, as wanted.
\par
It remains to prove the pullback formula for $b(p)$. Note that if we extend $b_1$ to the whole $P$ by setting $b_1(p)=0$ for $p\in P_{\rm{s}}$, then $b_1:P\to X^\gamma$ defines a sub-equilibrium and the mentioned formula is straightforward for $p\in P_{\rm{s}}$, since $b(p)=0$. For $p\in P_{\rm{f}}$,
once more we turn to the method given in the proof of Theorem~3.6 in Novo et al.~\cite{nono2} and we build the nondecreasing family  of sub-equilibria $(a_t)_{t\geq 0}:P\to X^\gamma$, $p\mapsto a_t(p):=u(t,p{\cdot}(-t),b_1(p{\cdot}(-t)))$. Since $a_t(p)\in A(p)$ for $t\geq 0$, there exists the limit $\lim_{t\to\infty} a_t(p)=\sup_{t\geq 0} a_t(p)=b_1^*(p)\in A(p)$ and it suffices to prove that $b_1^*(p)=b(p)$. For that, fix a $p\in P_{\rm{f}}$,  take any $(p,z)\in\A$ with $0\ll z<b(p)$ and let us check that $z\leq b_1^*(p)$. By Theorem~\ref{teor-A sublineal}~(i) there exists a $t_0>0$ such that $u(-t,p,z)$ remains in the linear zone for $t\geq t_0$ and thus $u(-t,p,z)\leq b_1(p{\cdot}(-t))$ for $t\geq t_0$. By monotonicity, then $z=u(t,p{\cdot}(-t),u(-t,p,z))\leq a_t(p)$ for $t\geq t_0$, and taking the supremum for $t\geq t_0$ we deduce that  $z\leq b_1^*(p)$, as we wanted to see. The proof is finished.
\end{proof}
\begin{teor}\label{teor-atraccion forward-sublineal}
Let $h\in\mathcal{U}(P\times\bar U)$ and assume that $g(p,x,y)$ satisfies $\rm{(c1)}$-$\rm{(c5)}$ plus the sublinear condition $\rm{(c6)}$. Then, for any $p_0\in P_{\rm{f}}$ with $\des\limsup_{t\to\infty} c(t,p_0)=\infty$,
\[
\lim_{t\to\infty} {\rm dist}(u(t,p_0,B),A(p_0{\cdot}t))=0 \quad\text{for any bounded set}\;\,B\subset X.
\]
\end{teor}
\begin{proof}
Fix a bounded set $B\subset X$. As in the proof of Theorem~\ref{teor-atraccion forward}, it is enough to see that given any $\varepsilon>0$ we can find a $t_*>0$ such that for any $t\geq t_*$ and for any $u(t,p_0,z)\in u(t,p_0,B)$ ($z\in B$) we can take a certain  $a(t,z) \in A(p_0{\cdot}t)$ such that $\|u(t,p_0,z)-a(t,z)\|\leq \varepsilon$ (see Remark~\ref{nota-forward atr} in the Dirichlet case).
\par
First of all, applying Theorem~\ref{teor-A sublineal}~(ii.2), we can recover assertion (a1) in the formerly mentioned proof, so that $u(t,p_0,B)\subset [-\lambda(t)\,b(p_0{\cdot}t),\lambda(t)\,b(p_0{\cdot}t)]$ for $t\geq 1$ and $\lambda(t)\downarrow 1$ as $t\to\infty$. We also recover  assertions (a2) with a constant $0<\rho<1$, and (a3) with constants $M_0>0,c_0\in\R$, assuming in the worst case that $c_0>0$.
\par
Given $0<\varepsilon<1$  and taking $m=\sup_{t\geq 0} \|b(p_0{\cdot}t)\|<\infty$, we can fix an $\varepsilon_0$ such that  $0<\varepsilon_0<\varepsilon$, $2\,\varepsilon_0<r_0$, $2\,\varepsilon_0\,m /(r_0-\varepsilon_0)<\varepsilon$ and $\varepsilon_0\,m <\varepsilon$.
This time following the proof of Theorem~\ref{teor-A en cono positivo}, we take a $\lambda_0>0$ such that $b(p)\leq \lambda_0\,e(p)$ for any $p\in P$, and the constants $E_0, C_0$ there introduced for need in the Dirichlet case, and we take a $T>1$ such that
\[
\frac{4\,\lambda_0\,E_0\,C_0\,M\,e^{-\delta t}c(t,p)}{\rho}\leq \varepsilon_0\quad\text{for }t\geq T-1,\;p\in P,
\]
where  $M, \delta>0$ are the constants in property (5) of the continuous separation (see Remark~\ref{nota-dirichlet} for the Dirichlet case).
\par
Then, as usual, associated to $T>1$, we fix a $\delta_0>0$ small enough so that $M_0\,e^{c_0 T}\delta_0<r_0$.
Since $\liminf_{t\to\infty} \|b(p_0{\cdot}t)\|=0$  by Proposition~\ref{prop-norma b}, having assertion (a1) in mind, we can take a $t_0>2$ sufficiently big so that  $\lambda(t)\leq 1+\varepsilon_0$ for any $t\geq t_0-1$,  and  $(1+\varepsilon_0)\,\|b(p_0{\cdot}(t_0-1))\|\leq \delta_0$, so that both $\|b(p_0{\cdot}(t_0-1))\|\leq \delta_0$ and  $\|u(t_0-1,p_0,z)\|\leq \delta_0$ for any $z\in B$.
\par
For each $z\in B$ we write for $t\geq -1$,
\[
u(t_0+t,p_0,z)=\alpha(t,z)\,e(p_0{\cdot}(t_0+t))+w(t,z)\in X_1(p_0{\cdot}(t_0+t))\oplus X_2(p_0{\cdot}(t_0+t))
\]
and just as in previous occasions we deduce from (a1) that $-2\,\lambda_0\leq \alpha(-1,z)\leq 2\,\lambda_0$ and $-4\,\lambda_0\,e(p_0{\cdot}(t_0-1))\leq w(-1,z)\leq 4\,\lambda_0\,e(p_0{\cdot}(t_0-1))$.
We know that for any $z\in B$ the solution $u(t_0-1+t,p_0,z)$ remains in the linear zone at least for $t\in [0,T]$. As in the proof of Theorem~\ref{teor-atraccion forward}, we start at time $t_0-1$ for technical reasons in the Dirichlet case, but now we look at the evolution from time $t_0$ on, so that for $t\geq 0$ and as far as the solution remains in the linear zone (at least for $t\in [0,T-1]$),
\begin{equation}\label{aux 3}
\begin{split}
u(t_0+t,p_0,z)&=u(t,p_0{\cdot}t_0,u(t_0,p_0,z))=\phi(t,p_0{\cdot}t_0)\,u(t_0,p_0,z)
\\&= \phi(t,p_0{\cdot}t_0)\,\alpha(0,z)\,e(p_0{\cdot}t)+ \phi(t,p_0{\cdot}t_0)\,w(0,z)
\\&= \alpha(0,z)\,c(t,p_0{\cdot}t_0)\,e(p_0{\cdot}(t_0+t))+ w(t,z)\,,
\end{split}
\end{equation}
and for $t\geq T-1$, while still in the linear zone, we can bound, using similar arguments to the ones used in the proof of Theorem~\ref{teor-A en cono positivo},
\[
\|w(t,z)\|_\gamma=\|\phi(t,p_0{\cdot}t_0)\,w(0,z)\|_\gamma\leq 4\,\lambda_0\,E_0\,C_0\,M\,e^{-\delta t} c(t,p_0{\cdot}t_0)\leq \varepsilon_0\,\rho<\varepsilon_0<\frac{r_0}{2}\,.
\]
Besides, by assertion (a2), provided that we keep in the linear zone,
\begin{equation}\label{aux}
-\varepsilon_0\,e(p_0{\cdot}(t_0+t))\leq  w(t,z)\leq \varepsilon_0\,e(p_0{\cdot}(t_0+t)) \quad \text{for}\; t\geq T-1\,.
\end{equation}
\par
We now distinguish three different cases depending on the position of the projection $\alpha(-1,z)\,e(p_0{\cdot}(t_0-1))$ of $u(t_0-1,p_0,z)$ onto $X_1(p_0{\cdot}(t_0-1))$ with respect to the threshold given by $b_1(p_0{\cdot}(t_0-1))=\gamma_1 e(p_0{\cdot}(t_0-1))$, for the map $b_1$ defined in~\eqref{b1} (see also~\eqref{b_1(p)}).
\par
{\it Case 1:} $\alpha(-1,z)=0$. It is easy to deduce from the bound for $\|w(t,z)\|_\gamma$ that in this case the solution $u(t_0-1+t,p_0,z)$ always remains in the linear zone for $t\geq 0$ and $\|u(t_0+t,p_0,z)\|=\|w(t,z)\|\leq \varepsilon_0 <\varepsilon$ for any $t\geq T-1$. Therefore, for any $t\geq t_0-1+T$ we can take $a(t,z)=0\in A(p_0{\cdot}t)$ such that $\|u(t,p_0,z)-a(t,z)\|\leq \varepsilon$.
\par
{\it Case 2:} $0<|\alpha(-1,z)|< \gamma_1$. We just consider the case $0<\alpha(-1,z)< \gamma_1$, since for negative $\alpha(-1,z)$ the arguments are just symmetric.  First of all, note that $0\ll \alpha(-1,z)\,e(p_0{\cdot}(t_0-1))< b_1(p_0{\cdot}(t_0-1))\leq b(p_0{\cdot}(t_0-1))$, so that by~\eqref{igualdad}, $\alpha(-1,z)\,e(p_0{\cdot}(t_0-1))\in A(p_0{\cdot}(t_0-1))$,  and $\|\alpha(-1,z)\,e(p_0{\cdot}(t_0-1))\|\leq \delta_0$  by the choice of $t_0$. In particular this means that the solution starting at $(p_0{\cdot}(t_0-1),\alpha(-1,z)\,e(p_0{\cdot}(t_0-1)))$ strictly remains in the linear zone at least for $t\in[0,T]$. Then, as long as both solutions $u(t_0-1+t,p_0,z)$ and $u(t,p_0{\cdot}(t_0-1),\alpha(-1,z)\,e(p_0{\cdot}(t_0-1)))$ remain in the linear zone (at least for $t\in[0,T]$) we can  take $a(t_0-1+t,z)=\phi(t,p_0{\cdot}(t_0-1))\,\alpha(-1,z)\,e(p_0{\cdot}(t_0-1))\in A(p_0{\cdot}(t_0-1+t))$ and for $t\geq T-1$, $\|u(t_0+t,p_0,z)-a(t_0+t,z)\|=\|w(t,z)\|\leq \varepsilon_0 <\varepsilon$ by~\eqref{aux}.
\par
Since $\limsup_{t\to\infty}c(t,p_0)=\infty$, from~\eqref{aux 3} and \eqref{aux} both solutions must escape, sooner or later, from the linear zone. Let us look at the first time $T_1>T$ such that one of them arrives at the border of the linear zone, and note that,  if $T_1$ were the same for both solutions, we could choose any of the following routes, which are just slightly different.
\par
First, let us assume  that $T_1>T$ is the first time such that $\alpha(T_1-1,z)=r_0$ and $\|u(t_0-1+t,p_0,z)\|<r_0$ for $t\in (0,T_1)$. For the sake of writing, take $t_1=T_1-1$.
Then, since the past semitrajectory of  $\alpha(t_1,z)\,e(p_0{\cdot}(t_0+t_1))$ lies in the linear zone and we have relation~\eqref{b_1(p)}, $b_1(p_0{\cdot}(t_0+t_1))=r_0\,e(p_0{\cdot}(t_0+t_1))$ and $(r_0-\varepsilon_0)\,e(p_0{\cdot}(t_0+t_1))\in A(p_0{\cdot}(t_0+t_1))$, and by~\eqref{aux} at $t=t_1$,
\[
(r_0-\varepsilon_0)\,e(p_0{\cdot}(t_0+t_1))\leq u(t_0+t_1,p_0,z)\leq (r_0+\varepsilon_0)\,e(p_0{\cdot}(t_0+t_1))\,.
\]
Applying monotonicity, and sublinearity in the second inequality,
we get, for $t\geq 0$,
\begin{align*}
u(t,p_0{\cdot}(t_0+t_1)&,(r_0-\varepsilon_0)\,e(p_0{\cdot}(t_0+t_1)))\leq u(t+t_0+t_1,p_0,z)\\&\leq \frac{r_0+\varepsilon_0}{r_0-\varepsilon_0}\,u(t,p_0{\cdot}(t_0+t_1),(r_0-\varepsilon_0)\,e(p_0{\cdot}(t_0+t_1)))\,,
\end{align*}
so that
\begin{align*}
0&\leq u(t+t_0+t_1,p_0,z)-u(t,p_0{\cdot}(t_0+t_1),(r_0-\varepsilon_0)\,e(p_0{\cdot}(t_0+t_1)))\\&\leq \frac{2\,\varepsilon_0}{r_0-\varepsilon_0}\,u(t,p_0{\cdot}(t_0+t_1),(r_0-\varepsilon_0)\,e(p_0{\cdot}(t_0+t_1)))\leq \frac{2\,\varepsilon_0}{r_0-\varepsilon_0}\,b(p_0{\cdot}(t+t_0+t_1))\,,
\end{align*}
and therefore, taking $a(t+t_0+t_1,z)=u(t,p_0{\cdot}(t_0+t_1),(r_0-\varepsilon_0)\,e(p_0{\cdot}(t_0+t_1)))\in A(p_0{\cdot}(t+t_0+t_1))$, we conclude that
\begin{equation}\label{aux 2}
\|u(t+t_0+t_1,p_0,z)-a(t+t_0+t_1,z)\|\leq \frac{2\,\varepsilon_0}{r_0-\varepsilon_0}\,m\leq \varepsilon\,.
\end{equation}
\par
For the second situation, let us assume that $T_1>T$ is the first time such that  $\|u(t_0-1+T_1,p_0,z)\|=r_0$ and $\alpha(-1+t,z)<r_0$ for $t\in (0,T_1)$. Again, take $t_1=T_1-1$ for the sake of writing, so that $\|u(t_0+t_1,p_0,z)\|=r_0$. Since by~\eqref{aux},
\[
(\alpha(t_1,z)-\varepsilon_0)\,e(p_0{\cdot}(t_0+t_1))\leq u(t_0+t_1,p_0,z)\leq (\alpha(t_1,z)+\varepsilon_0)\,e(p_0{\cdot}(t_0+t_1))\,,
\]
it must be $r_0-\varepsilon_0\leq \alpha(t_1,z)\leq  r_0$. Also now $(r_0-\varepsilon_0)\,e(p_0{\cdot}(t_0+t_1))\in A(p_0{\cdot}(t_0+t_1))$ since the past semitrajectory of  $\alpha(t_1,z)\,e(p_0{\cdot}(t_0+t_1))$ lies in the linear zone. Then,
\[
(r_0-2\,\varepsilon_0)\,e(p_0{\cdot}(t_0+t_1))\leq u(t_0+t_1,p_0,z)\leq (r_0+\varepsilon_0)\,e(p_0{\cdot}(t_0+t_1))\,
\]
and applying the monotonicity and sublinearity of the semiflow, for $t\geq 0$,
\begin{align*}
\frac{r_0-2\,\varepsilon_0}{r_0-\varepsilon_0}\,u(t,p_0{\cdot}(t_0+t_1),&(r_0-\varepsilon_0)\,e(p_0{\cdot}(t_0+t_1)))\leq u(t+t_0+t_1,p_0,z)\\&\leq \frac{r_0+\varepsilon_0}{r_0-\varepsilon_0}\,u(t,p_0{\cdot}(t_0+t_1),(r_0-\varepsilon_0)\,e(p_0{\cdot}(t_0+t_1)))\,.
\end{align*}
For $t\geq 0$, we can take $a(t+t_0+t_1,z)\in A(p_0{\cdot}(t+t_0+t_1))$ the same as before, to conclude from here that~\eqref{aux 2} also holds. Summing up, in the so-called Case 2 we can take $t_*=t_0-1+T$.
\par
{\it Case 3:} $|\alpha(-1,z)|\geq \gamma_1$. Once more, we just consider the case $\gamma_1\leq \alpha(-1,z)$, since for negative $\alpha(-1,z)$ the arguments are similar. In this situation the candidate for $a(t,z) \in A(p_0{\cdot}t)$ for $t\geq t_*$ ($t_*$ to be determined) is given by $b(p_0{\cdot}t)$, as we are going to check. On this occasion we have
\begin{align*}
b_1(p_0{\cdot}(t_0-1))+w(-1,z)&\leq \alpha(-1,z)\,e(p_0{\cdot}(t_0-1)+w(-1,z)\\&=u(t_0-1,p_0,z)
\leq (1+\varepsilon_0)\,b(p_0{\cdot}(t_0-1))\,,
\end{align*}
and we only have to care about the lower bound, since by monotonicity and sublinearity,  for $t\geq 0$,
$u(t_0-1+t,p_0,z)\leq (1+\varepsilon_0)\,b(p_0{\cdot}(t_0-1+t))$.
\par
Once again by the choice of $t_0$, $\|b_1(p_0{\cdot}(t_0-1))\|\leq \delta_0$ so that also the solution starting at $(p_0{\cdot}(t_0-1),b_1(p_0{\cdot}(t_0-1)))=(p_0{\cdot}(t_0-1),\gamma_1 e(p_0{\cdot}(t_0-1)))$ remains strictly in the linear zone at least for $t\in [0,T]$. Having present that $\A\subset P\times X^\gamma$ is compact (thus, bounded), and assertion (a2) in the proof of Theorem~\ref{teor-atraccion forward}, we can find a sufficiently small $0<\delta<1$ so that
 for any $t\geq 0$,
\begin{equation}\label{delta}
\delta\,b(p_0{\cdot}t)\leq (r_0-2\,\varepsilon_0)\,e(p_0{\cdot}t)\,.
\end{equation}
Let us take the first $T_2\geq T$ such that $\gamma_1 c(T_2,p_0{\cdot}(t_0-1))=r_0$ and write $t_2=T_2-1$.
Now, it might happen that $\|u(t_0-1+t,p_0,z)\|<r_0$ for $t\in (0,T_2)$. Then, by~\eqref{aux} and~\eqref{delta},
$\delta\,b(p_0{\cdot}(t_0+t_2))\leq (r_0-\varepsilon_0)\,e(p_0{\cdot}(t_0+t_2))\leq u(t_0+t_2,p_0,z)$.
But it might also happen that $\|u(t_0-1+T_1,p_0,z)\|=r_0$ for some $T_1\in (T,T_2]$. In this case, for $t_1=T_1-1$ necessarily $\alpha(t_1,z)\geq r_0-\varepsilon_0$ and then,
\[
\delta\,b(p_0{\cdot}(t_0+t_1))\leq(r_0-2\,\varepsilon_0)\,e(p_0{\cdot}(t_0+t_1))\leq u(t_0+t_1,p_0,z)\,,
\]
and we can use the monotonicity and sublinearity of the semiflow to get that also in this case $\delta\,b(p_0{\cdot}(t_0+t_2))\leq u(t_0+t_2,p_0,z)$.
\par
To finish, since $0< \delta\,b(p_0{\cdot}(t_0+t_2))\leq b(p_0{\cdot}(t_0+t_2))$, by Theorem~\ref{teor-A sublineal}~(ii.1) we can assert that there exists a $t_3>0$ such that for $t\geq t_3$,
\[
(1-\varepsilon_0)\,b(p_0{\cdot}(t_0+t_2+t))\leq u(t,p_0{\cdot}(t_0+t_2),\delta\,b(p_0{\cdot}(t_0+t_2)))\leq u(t_0+t_2+t,p_0,z)\,,
\]
and for $t\geq t_0+t_2+t_3$ we can take $a(t,z)=b(p_0{\cdot}t)\in A(p_0{\cdot}t)$ so that $(1-\varepsilon_0)\,b(p_0{\cdot}t)\leq u(t,p_0,z)\leq (1+\varepsilon_0)\,b(p_0{\cdot}t)$ and $\|u(t,p_0,z)-a(t,z)\|\leq \varepsilon_0\,m\leq \varepsilon$.
\par
Summing up, in all the three cases we can  choose the biggest $t_*$ which is $t_*=t_0+t_2+t_3$, so that for $t\geq t_*$  and $z\in B$ we can take the indicated $a(t,z)\in A(p_0{\cdot}t)$ in each case  so that   $\|u(t,p_0,z)-a(t,z)\|\leq \varepsilon$. The proof is finished.
\end{proof}

\end{document}